\documentclass[EJP]{ejpecp}
\usepackage{enumitem}  
\usepackage{tikz}
\usepackage{ifthen}

\SHORTTITLE{Local limit of the fixed point forest}

\TITLE{Local limit of the fixed point forest\thanks{TJ was partially supported by NSF grant DMS--1401479.
AS was partially supported by NSF grants OCI--1147247 and DMS--1500050.}
}

\AUTHORS{%
  Name~Surname\footnote{Address\EMAIL{electronic-address}}%
}
\AUTHORS{%
  Tobias~Johnson\footnote{New York University.
    \EMAIL{tobias.johnson@nyu.edu} \url{http://cims.nyu.edu/\~tjohnson}}
  \and 
  Anne~Schilling\footnote{University of California, Davis.
    \EMAIL{anne@math.ucdavis.edu} \url{http://www.math.ucdavis.edu/\~anne}}
    \and
  Erik~Slivken\footnote{University of California, Davis.
    \EMAIL{erikslivken@math.ucdavis.edu}
    \url{http://www.math.ucdavis.edu/\~erikslivken}}
    }


\AMSSUBJ{60J80} 
\AMSSUBJSECONDARY{05C05; 05A05; 60C05; 60F99} 

\SUBMITTED{July 20, 2016}
\ACCEPTED{February 4, 2017}

\ARXIVID{1605.09777}

\VOLUME{22}
\YEAR{2017}
\PAPERNUM{18}
\DOI{10.1214/17-EJP36}

\ABSTRACT{Consider the following partial ``sorting algorithm'' on permutations: take the first entry of the permutation in one-line notation
and insert it into the position of its own value. Continue until the first entry is $1$. This process imposes a forest structure on
the set of all permutations of size $n$, where the roots are the permutations starting with $1$ and the leaves
are derangements. Viewing the process in the opposite direction towards the leaves, one picks a fixed point
and moves it to the beginning. Despite its simplicity, this ``fixed point forest'' exhibits a rich structure.
In this paper, we consider the fixed point forest in the limit $n\to \infty$ and show using Stein's method that at
a random permutation the local structure weakly converges to a tree defined in terms of independent Poisson point processes.
We also show that the distribution of the length of the longest path from a random permutation
to a leaf converges to the geometric distribution with mean $e-1$, and the length of the shortest path converges to the
Poisson distribution with mean $1$. In addition, the higher moments are bounded and hence the expectations converge
as well.}

\definecolor{darkred}{rgb}{0.7,0,0} 
\newcommand{\defn}[1]{{\color{darkred}\emph{#1}}} 
\newcommand{\red}[1]{{\color{red}#1}}
\newcommand{\ppp}[2][]{\xi_{#2}\ifthenelse{\equal{#1}{}}{}{^{#1}}}
\newcommand{\abs}[1]{\lvert #1 \rvert}
\newcommand{\ZZ}{\mathbb{Z}}
\newcommand{\Poi}{\mathrm{Poi}}
\newcommand{\Unif}{\mathrm{Unif}}
\renewcommand{\P}{\mathbf{P}}
\newcommand{\E}{\mathbf{E}}
\newcommand{\cov}{\mathop{\mathbf{Cov}}\nolimits}
\newcommand{\pppp}[1][]{\xi\ifthenelse{\equal{#1}{}}{}{^{#1}}}
\newcommand{\I}{\mathbf{I}}
\newcommand{\Ii}{\mathcal{I}}
\newcommand{\omitted}{\bullet}
\newcommand{\dtv}{d_{TV}}
\newcommand{\Z}{\mathbf{Z}}
\newcommand{\Y}{\mathbf{Y}}
\newcommand{\pp}[1]{\xi^{(n)}_{#1}}
\newcommand{\ppr}[1]{\overline{\xi}^{(n)}_{#1}}
\newcommand{\step}[2]{\par\medskip\par\noindent \textbf{Step #1}.
  \textit{#2}}
\newcommand{\1}{\mathbf{1}}
\newcommand{\toL}{\,{\buildrel d \over \longrightarrow}\,}
\newcommand{\Geo}{\mathrm{Geo}}
\newcommand{\Exp}{\mathrm{Exp}}
\newcommand{\B}{\mathcal{B}}

\begin{document}


  \section{Introduction}

\subsection{Fixed point forest}

Consider a deck of $n$ cards labeled $1,2,\ldots, n$ given in an arbitrary order. Take the top card and reinsert it
into the pile at the position of its value. This gives rise to a partial sorting algorithm, where the algorithm stops
when card 1 is on the top. This can be formulated in terms of the set $\mathfrak{S}_n$ of permutations of size
$n$. Each permutation $\pi$ represents an ordered deck of cards and can be written in one-line notation as
$\pi(1) \pi(2) \ldots \pi(n)$, where $\pi(i)$ is the value of the $i$-th card in the deck with $\pi(1)$ being the topmost
card. Then $\pi(1)$ is removed and reinserted into the position $\pi(1)$. For example, the permutation $3142
\in \mathfrak{S}_4$ goes to $1432$. This defines a graph with vertices being all permutations in $\mathfrak{S}_n$
and vertex $\pi$ being connected by an edge to vertex $\pi'$, if $\pi'$ is obtained from $\pi$ by moving the first entry
$\pi(1)$ to position $\pi(1)$. It was shown in~\cite{McKinley} that the graph is in fact a rooted forest, which we call $F_n$.
 A \defn{rooted forest} is a union of \defn{rooted trees}, where a tree is a graph that does not contain
any closed loops involving distinct vertices. Since later on we will consider trees with random permutations as roots,
we call each root of $F_n$ a \defn{base}. The bases of the forest are the permutations with $\pi(1)=1$,
and the leaves are the derangements, that is, permutations without fixed points. The partial sorting algorithm describes
a path from a permutation to a base.
Viewing the process in the opposite direction towards the leaves (i.e. vertices without incoming edge), one picks a fixed
point $\pi(i)=i$ of the permutation and moves it to the beginning, which we call \defn{bumping} the fixed point~$i$. For this reason, we call this forest the \defn{fixed point forest}.
Examples for $n=3$ and $n=4$ are given in Figures~\ref{figure.forest3} and~\ref{figure.forest4}, respectively.

\begin{figure}[t]
  \centering
\scalebox{0.8}{
\begin{tikzpicture}[>=latex,line join=bevel,]
\node (node_5) at (72.0bp,117.0bp) [draw,draw=none] {$321$};
  \node (node_4) at (18.0bp,63.0bp) [draw,draw=none] {$312$};
  \node (node_3) at (72.0bp,171.0bp) [draw,draw=none] {$231$};
  \node (node_2) at (72.0bp,63.0bp) [draw,draw=none] {$213$};
  \node (node_1) at (0.0bp,9.0bp) [draw,draw=none] {$132$};
  \node (node_0) at (45.0bp,9.0bp) [draw,draw=none] {$123$};
  \definecolor{strokecolor}{rgb}{0.0,1.0,0.0};
  \draw [strokecolor,->] (node_4) ..controls (26.154bp,46.691bp) and (31.523bp,35.953bp)  .. (node_0);
  \definecolor{strokecolor}{rgb}{0.0,1.0,0.0};
  \draw [strokecolor,->] (node_5) ..controls (72.0bp,101.01bp) and (72.0bp,90.861bp)  .. (node_2);
  \definecolor{strokecolor}{rgb}{0.0,0.0,1.0};
  \draw [strokecolor,->] (node_2) ..controls (63.846bp,46.691bp) and (58.477bp,35.953bp)  .. (node_0);
  \definecolor{strokecolor}{rgb}{0.0,0.0,1.0};
  \draw [strokecolor,->] (node_3) ..controls (72.0bp,155.01bp) and (72.0bp,144.86bp)  .. (node_5);
\end{tikzpicture}}
\caption{Fixed point forest $F_3$. \label{figure.forest3}}
\end{figure}
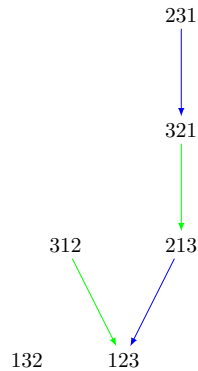

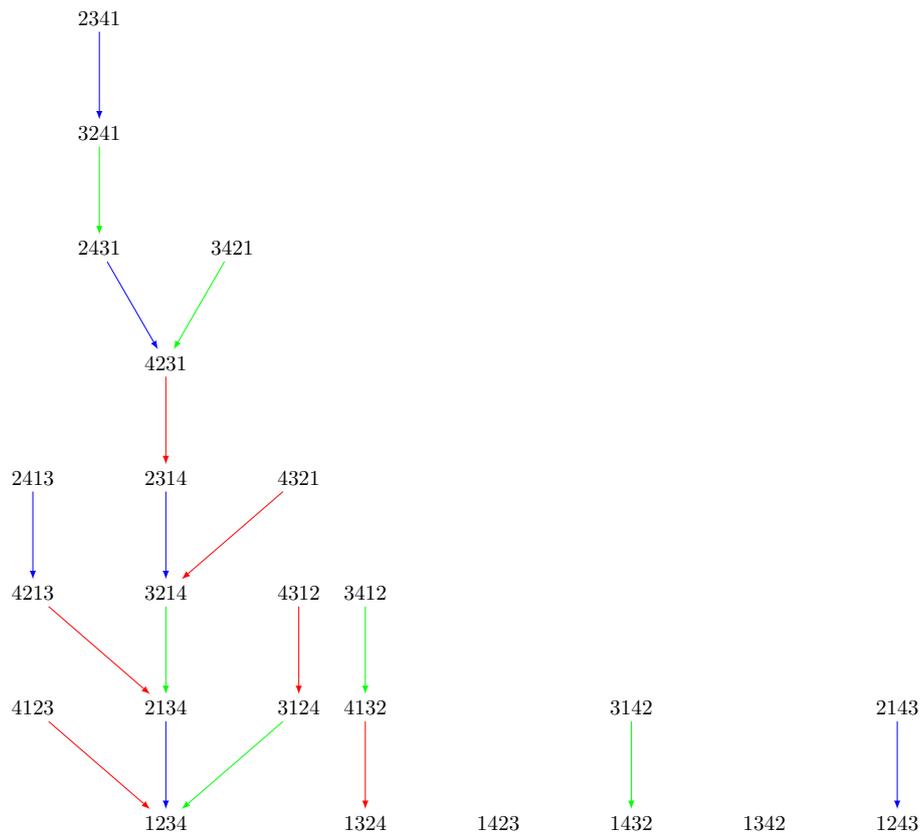
\begin{figure}[t]
\centering\scalebox{0.8}{
\begin{tikzpicture}[>=latex,line join=bevel,]
\node (node_22) at (146.0bp,117.0bp) [draw,draw=none] {$4312$};
  \node (node_23) at (146.0bp,171.0bp) [draw,draw=none] {$4321$};
  \node (node_20) at (22.0bp,117.0bp) [draw,draw=none] {$4213$};
  \node (node_21) at (84.0bp,225.0bp) [draw,draw=none] {$4231$};
  \node (node_9) at (53.0bp,387.0bp) [draw,draw=none] {$2341$};
  \node (node_8) at (84.0bp,171.0bp) [draw,draw=none] {$2314$};
  \node (node_7) at (425.0bp,63.0bp) [draw,draw=none] {$2143$};
  \node (node_6) at (84.0bp,63.0bp) [draw,draw=none] {$2134$};
  \node (node_5) at (301.0bp,9.0bp) [draw,draw=none] {$1432$};
  \node (node_4) at (239.0bp,9.0bp) [draw,draw=none] {$1423$};
  \node (node_3) at (363.0bp,9.0bp) [draw,draw=none] {$1342$};
  \node (node_2) at (177.0bp,9.0bp) [draw,draw=none] {$1324$};
  \node (node_1) at (425.0bp,9.0bp) [draw,draw=none] {$1243$};
  \node (node_0) at (84.0bp,9.0bp) [draw,draw=none] {$1234$};
  \node (node_19) at (177.0bp,63.0bp) [draw,draw=none] {$4132$};
  \node (node_18) at (22.0bp,63.0bp) [draw,draw=none] {$4123$};
  \node (node_17) at (115.0bp,279.0bp) [draw,draw=none] {$3421$};
  \node (node_16) at (177.0bp,117.0bp) [draw,draw=none] {$3412$};
  \node (node_15) at (53.0bp,333.0bp) [draw,draw=none] {$3241$};
  \node (node_14) at (84.0bp,117.0bp) [draw,draw=none] {$3214$};
  \node (node_13) at (301.0bp,63.0bp) [draw,draw=none] {$3142$};
  \node (node_12) at (146.0bp,63.0bp) [draw,draw=none] {$3124$};
  \node (node_11) at (53.0bp,279.0bp) [draw,draw=none] {$2431$};
  \node (node_10) at (22.0bp,171.0bp) [draw,draw=none] {$2413$};
  \definecolor{strokecolor}{rgb}{1.0,0.0,0.0};
  \draw [strokecolor,->] (node_22) ..controls (146.0bp,101.01bp) and (146.0bp,90.861bp)  .. (node_12);
  \definecolor{strokecolor}{rgb}{0.0,0.0,1.0};
  \draw [strokecolor,->] (node_7) ..controls (425.0bp,47.006bp) and (425.0bp,36.861bp)  .. (node_1);
  \definecolor{strokecolor}{rgb}{0.0,0.0,1.0};
  \draw [strokecolor,->] (node_10) ..controls (22.0bp,155.01bp) and (22.0bp,144.86bp)  .. (node_20);
  \definecolor{strokecolor}{rgb}{0.0,0.0,1.0};
  \draw [strokecolor,->] (node_9) ..controls (53.0bp,371.01bp) and (53.0bp,360.86bp)  .. (node_15);
  \definecolor{strokecolor}{rgb}{1.0,0.0,0.0};
  \draw [strokecolor,->] (node_18) ..controls (41.127bp,46.341bp) and (54.937bp,34.313bp)  .. (node_0);
  \definecolor{strokecolor}{rgb}{0.0,1.0,0.0};
  \draw [strokecolor,->] (node_17) ..controls (105.59bp,262.61bp) and (99.342bp,251.72bp)  .. (node_21);
  \definecolor{strokecolor}{rgb}{0.0,0.0,1.0};
  \draw [strokecolor,->] (node_6) ..controls (84.0bp,47.006bp) and (84.0bp,36.861bp)  .. (node_0);
  \definecolor{strokecolor}{rgb}{0.0,0.0,1.0};
  \draw [strokecolor,->] (node_11) ..controls (62.408bp,262.61bp) and (68.658bp,251.72bp)  .. (node_21);
  \definecolor{strokecolor}{rgb}{0.0,1.0,0.0};
  \draw [strokecolor,->] (node_12) ..controls (126.87bp,46.341bp) and (113.06bp,34.313bp)  .. (node_0);
  \definecolor{strokecolor}{rgb}{1.0,0.0,0.0};
  \draw [strokecolor,->] (node_23) ..controls (126.87bp,154.34bp) and (113.06bp,142.31bp)  .. (node_14);
  \definecolor{strokecolor}{rgb}{1.0,0.0,0.0};
  \draw [strokecolor,->] (node_20) ..controls (41.127bp,100.34bp) and (54.937bp,88.313bp)  .. (node_6);
  \definecolor{strokecolor}{rgb}{0.0,1.0,0.0};
  \draw [strokecolor,->] (node_15) ..controls (53.0bp,317.01bp) and (53.0bp,306.86bp)  .. (node_11);
  \definecolor{strokecolor}{rgb}{0.0,1.0,0.0};
  \draw [strokecolor,->] (node_14) ..controls (84.0bp,101.01bp) and (84.0bp,90.861bp)  .. (node_6);
  \definecolor{strokecolor}{rgb}{1.0,0.0,0.0};
  \draw [strokecolor,->] (node_19) ..controls (177.0bp,47.006bp) and (177.0bp,54.86bp)  .. (node_2);
  \definecolor{strokecolor}{rgb}{0.0,0.0,1.0};
  \draw [strokecolor,->] (node_8) ..controls (84.0bp,155.01bp) and (84.0bp,144.86bp)  .. (node_14);
  \definecolor{strokecolor}{rgb}{0.0,1.0,0.0};
  \draw [strokecolor,->] (node_13) ..controls (301.0bp,47.006bp) and (301.0bp,36.861bp)  .. (node_5);
  \definecolor{strokecolor}{rgb}{0.0,1.0,0.0};
  \draw [strokecolor,->] (node_16) ..controls (177.0bp,101.01bp) and (177.0bp,90.861bp)  .. (node_19);
  \definecolor{strokecolor}{rgb}{1.0,0.0,0.0};
  \draw [strokecolor,->] (node_21) ..controls (84.0bp,209.01bp) and (84.0bp,198.86bp)  .. (node_8);
\end{tikzpicture}}
\caption{Fixed point forest $F_4$. \label{figure.forest4}}
\end{figure}

The goal of this paper is to understand the local structure of this forest as $n\to\infty$. Put differently,
we would like to describe the neighborhood of a ``typical'' permutation in the forest for large $n$.
We carry this out in Theorem~\ref{thm:lwc}, where we find the limit of the forest as $n\to\infty$
in the sense of \defn{local weak convergence},
as defined in Section~\ref{sec:lwc.def} (see \cite{AS} for more background). The limit is a random tree
that we construct from an infinite collection of independent Poisson point processes. The
proof relies on Stein's method for Poisson approximation.
The limit tree seems to have interesting properties; see
Section~\ref{sec:comparison} for comparisons to other processes. As a corollary of this local
weak convergence,
we show that the distributions of the distance from a random permutation to the farthest and nearest
leaves descending from it converge, respectively, to a geometric distribution with
mean $e-1$ and a Poisson distribution with mean $1$.
With some additional work, we prove that the higher moments
are bounded and hence the expectations converge as well (see Theorem~\ref{thm:limits}).

We were first made aware of the fixed point forest by Gwen McKinley who studied it in her undergraduate
thesis~\cite{McKinley}, which she wrote under the guidance of the second author. McKinley was introduced to the
problem at the Missouri State University summer REU program by Les Reid to whom the process was suggested
by Gerhardt Hinkle. In her thesis, McKinley investigated several global properties of the forest, finding
for instance that the longest path to the base is of length $2^{n-1}-1$.
Despite the simplicity of the description of the fixed point forest, many basic questions about it
seem difficult and remain unanswered. We discuss some of them in more detail in Section~\ref{sec:further}.
Both locally and globally, the structure of the fixed point forest seems quite rich.

In the remainder of the introduction, we provide a non-technical discussion of the local structure of the forest $F_n$
in the limit $n\to \infty$. Let $\pi_n$ be a uniformly random permutation in $F_n$.
The essential information for determining the neighborhood of $\pi_n$ is the location
of its fixed and near-fixed points. Our idea for understanding the limiting structure of $F_n$
is to construct a sort of limit of the fixed and near-fixed points of $\pi_n$. By rescaling by
a factor of $1/n$, these are represented as Poisson point processes on $[0,1]$.
Then, we define a tree from these point processes, which will turn out to be the local limit of $F_n$
(see Section~\ref{sec:local.weak.convergence}).

  \subsection{Moving towards leaves in the tree}
  Suppose that $\pi$ is a permutation and that we would like to enumerate its descendants up to three
  levels in the forest; that is, we want to determine all permutations obtained from $\pi$ by bumping
  fixed points to the beginning no more than three times.
  What information about $\pi$ do we need?

  The answer is that we must know all $i$ such that $\pi(i)=i$, $\pi(i)=i+1$, or $\pi(i)=i+2$.
  This is best seen by example. To unify our terminology, say that the letter~$\pi(i)$ or the position~$i$
  is \defn{$k$-separated} if $\pi(i)=i+k$. A $0$-separated letter is simply a fixed point.

  \begin{example}\label{ex:1}
    Suppose that $\pi$ has $0$-separated letters at positions~$7$ and~$27$, that it has
    a $1$-separated letter at $18$, and that it has a $2$-separated letter at $13$.
    Then $\pi$ has two children in the forest, given by bumping the letters in either position~$7$ or~$27$.

    If the letter in position~$7$ is bumped, then the resulting permutation still has a $0$-separated letter
    at $27$, a $1$-separated letter at $18$, and a $2$-separated letter at $13$. From here, one child
    is given by bumping the letter in position~$27$. Then, this turns the $1$-separated letter at position $18$ into a
    $0$-separated letter at $19$, and it turns the $2$-separated letter at $13$ into a $1$-separated
    letter at $14$. From here, there is a child given by bumping the letter in position~$19$ (as well as another
    child after that).

    If the letter in position~$27$ is bumped in the first step, then the $0$-separated letter at $7$ is destroyed,
    the $1$-separated letter at $18$ becomes a $0$-separated letter at $19$, and the $2$-separated
    letter becomes a $1$-separated letter at $14$. Now there is a child given by bumping the letter in position~$19$,
    which turns the $1$-separated letter at $14$ into a $0$-separated letter at $15$, which can then be bumped.
  \end{example}

  This example suggests a way to encode a permutation according to its $k$-separated letters
  for $k=0,\ldots,K$. We create a word containing a $k$ for each $k$-separated position in the permutation,
  in the order that they appear.
  For example, with $K=2$, the word corresponding to $\pi$ in Example~\ref{ex:1} is
  $0210$. When we bump the $0$-separated letter at a given position, we remove the $0$ from this
  position, subtract $1$ from all letters previous to the bumped position, and we leave alone all letters after the
  bumped position. We write a $-k$ letter as $\red{\bar{k}}$, that is, barred and colored red.
  These negative letters are irrelevant for determining the descendants
  of a permutation, but we leave them in the word for consistency with the next section.
  Thus bumping the first $0$ from $0210$ yields $210$,
  and bumping the second $0$ yields $\red{\bar{1}}10$. This gives us the compact
  depiction of the descendants of $\pi$ in Figure~\ref{fig:ex1}.

  We warn the reader that this picture is incomplete in one way. When a $0$-separated letter at position~$i$
  is bumped, it creates a new $(i-1)$-separated letter at position~$1$.
  In Example~\ref{ex:1}, this makes no difference, but in the following example it does.
  \begin{figure}
    \begin{center}
      \begin{tikzpicture}[yscale=0.9]
        \path (0,0) node (pi) {$0210$}
              (-1,1) node (c1) {$210$}
              (1,1) node (c2) {$\red{\bar{1}}10$}
              (-1,2) node (c11) {$10$}
              (-1,3) node (c111) {$0$}
              (1,2) node(c21) {$\red{\bar{2}}0$}
              (1,3) node (c211) {$\red{\bar{3}}$}
            ;
            \draw (pi.north)+(-.38,0)--(c1) (pi.north)+(.4,0)--(c2)
                  (c1.north)+(.13,0)--(c11) (c11.north)+(.07,0)--(c111) (c2.north)+(.13,0)--(c21)
                  (c21.north)+(0.07,0)--(c211)
              ;
      \end{tikzpicture}
    \end{center}
    \caption{The descendants of $\pi$ from Example~\ref{ex:1},
    encoded by their words. Only three levels are given here, since we are using
    the information only from the $0$-, $1$-, and $2$-separated letters.
    To construct a tree like this given only the encoding of $\pi$ by its word,
    when we bump the $0$-separated letter at a given position,
    we remove the $0$ from this position and subtract $1$ from all
    letters previous to the bumped position. A value of $-k$ is indicated here
    by $\red{\bar{k}}$, that is, barred and colored red.}\label{fig:ex1}
  \end{figure}
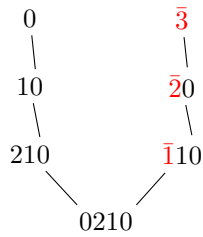

  \begin{example}
    Suppose that $\pi$ is the permutation $42135$, in one-line notation.
    Then $\pi$ has $0$-separated letters at locations $2$ and $5$, and it has no $1$- or $2$-separated
    letters. If $5$ is bumped, the resulting permutation
    has no $0$-separated letters and hence no children.
    If $2$ is bumped, then we still have a $0$-separated letter at $5$.
    We also have a new $1$-separated letter at position~$1$, namely the $2$ that was just bumped.
    Thus, if $5$ is bumped, we can then bump the $2$ again.

    If we start with the word $00$ corresponding to $42135$ and then follow the rules laid out
    before this example for manipulating these words, we miss this last descendant
    (see Figure~\ref{fig:ex2}).
  \end{example}
  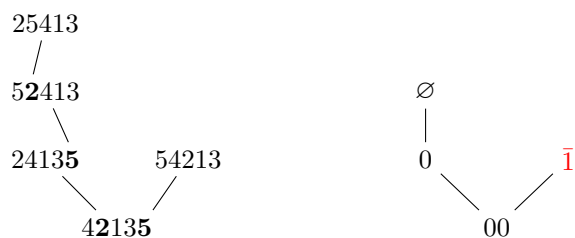
\begin{figure}
    \begin{center}
      \begin{tikzpicture}[xscale=1.25,yscale=0.9]
        \begin{scope}
          \path (0,0) node (pi) {$4\mathbf{2}13\mathbf{5}$}
              (-0.75,1) node (c1) {$2413\mathbf{5}$}
              (0.75,1) node (c2) {$54213$}
              (-0.75,2) node (c11) {$5\mathbf{2}413$}
              (-0.75,3) node (c111) {$25413$}
            ;
          \draw (pi.north)+(-.22,0)--(c1) (pi.north)+(.38,0)--(c2)
                  (c1.north)+(.24,0)--(c11) (c11.north)+(-.13,0)--(c111);
        \end{scope}
        \begin{scope}[shift={(4,0)}]
          \path (0,0) node (pi) {$00$}
              (-0.75,1) node (c1) {$0$}
              (0.75,1) node (c2) {$\red{\bar{1}}$}
              (-0.75,2) node (c11) {$\varnothing$}
            ;
            \draw (pi)--(c1) (pi)--(c2)
                  (c1)--(c11);
        \end{scope}

      \end{tikzpicture}
    \end{center}

    \caption{On the left, we depict the descendants of the permutation $42135$ in the forest.
    Fixed points are in bold.
    On the right, we show the apparent forest computed only using the words giving the order
    of the $k$-separated letters. This misses a descendant created by the ``reentry'' of $2$ as a fixed
    point.}\label{fig:ex2}
  \end{figure}

  \label{page:reentry}
  This possible ``reentry''  of a bumped letter as a fixed point complicates the picture.
  When we determine the descendants of a permutation~$\pi$ up to level~$K$, this reentry can only occur
  if there is a $0$-separated letter at one of positions~$1,\ldots,K-1$. For fixed~$K$, this is vanishingly
  unlikely as $n\to\infty$, and so in constructing the limit tree, we ignore it.
  The essential idea to this construction is to take a random word on the alphabet $0,\ldots,K$
  and make a tree by the procedure described in
  Figure~\ref{fig:ex1}. In constructing the limit tree, one complication is that we want $K$ to
  be infinite. To address this, we represent the locations of $k$-separated letters for each $k$
  in an abstracted permutation~$\pi$ as a point process $\ppp[\pi]{k}$,
  representing a set of locations on the interval~$[0,1]$.
  These will be \defn{independent Poisson point processes} with intensity one in the limit.
  For any fixed~$K$, we can then obtain a string by writing a $k$ for each point of $\ppp[\pi]{k}$
  for $0\leq k< K$, sorted by the positions of the points in $[0,1]$. This string then determines
  a tree up to level~$K$ following the procedure sketched out after Example~\ref{ex:1}.
  \begin{example}
    Suppose $\ppp[\pi]{0}$, $\ppp[\pi]{1}$, and $\ppp[\pi]{2}$ contain points as depicted in
    Figure~\ref{fig:ex3}. Then the word associated with the abstracted permutation~$\pi$ for $K=2$
    is $0120$, and the tree generated by it up to three levels is the same one as in Example~\ref{ex:1}
    and Figure~\ref{fig:ex1}.
  \end{example}
  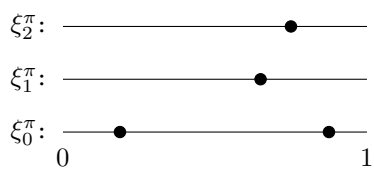
\begin{figure}
    \begin{center}
      \begin{tikzpicture}[vert/.style={circle,fill,inner sep=0,
            minimum size=0.15cm,draw},yscale=.7]
        \draw[{[-]}] (0,0)--(4,0);
        \draw[{[-]}] (0,1)--(4,1);
        \draw[{[-]}] (0,2)--(4,2);
        \path (.75,0) node[vert] {}
              (3.5, 0) node[vert] {}
              (2.6, 1) node[vert] {}
              (3, 2) node[vert] {};
        \path (0,0) node [label=left:{$\ppp[\pi]{0}$:},label=below:$0$] {}
              (4,0) node[label=below:$1$] {}
              (0,1) node [label=left:{$\ppp[\pi]{1}$:}] {}
              (0,2) node [label=left:{$\ppp[\pi]{2}$:}] {};
      \end{tikzpicture}
    \end{center}
    \caption{The point processes $\ppp[\pi]{0}$, $\ppp[\pi]{1}$, and $\ppp[\pi]{2}$ on the interval
    $[0,1]$, representing the $0$-, $1$-, and $2$-separated letters in the abstracted permutation~$\pi$.
    The word associated with this permutation for $K=2$ is $0120$, and the tree of its descendants
    up to three levels is the same as in Figure~\ref{fig:ex1}.}
    \label{fig:ex3}
  \end{figure}

  One can also construct the tree directly from the point processes, without the intermediate step
  of converting to a word. The abstracted permutation~$\pi$ with associated point processes
  $\ppp[\pi]{0},\ppp[\pi]{1},\ldots$ has $\abs{\ppp[\pi]{0}}$ children, with $\abs{\ppp[\pi]{0}}$ denoting
  the number of points in $\ppp[\pi]{0}$. The point processes obtained by bumping a point $x\in[0,1]$
  of $\ppp[\pi]{0}$ are given
  by removing the point $x$ from $\ppp[\pi]{0}$, and by ``shifting down'' each point process
  on $[0,x)$ as in Figure~\ref{fig:down.shift}.\looseness=1
  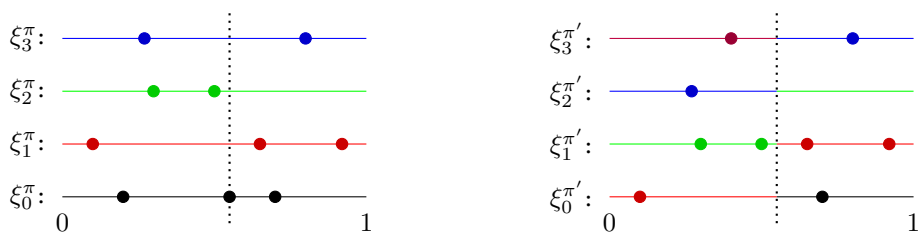
\begin{figure}\centering
      \begin{tikzpicture}[vert/.style={circle,fill,inner sep=0,
            minimum size=0.15cm,draw},yscale=.7,xscale=4]
        \begin{scope}
          \draw[{[-]}] (0,0)--(1,0);
          \draw[{[-]}, red] (0,1)--(1,1);
          \draw[{[-]}, green] (0,2)--(1,2);
          \draw[{[-]},blue] (0,3)--(1,3);
          \draw[dotted,thick] (.55,-.5) -- (.55,3.5);
          \path (.2,0) node[vert] {} (.55,0) node[vert] {};
          \path (.8,3) node[vert,blue!80!black] {} (.92,1) node[vert,red!80!black] {}
                (.7,0) node[vert] {} (.65,1) node[vert,red!80!black] {};
          \path (.1,1) node[vert,red!80!black] {} (.3,2) node[vert,green!80!black] {}
                (.5,2) node[vert,green!80!black] {}  (.27,3) node[vert,blue!80!black] {};
          \path (0,0) node [label=left:{$\ppp[\pi]{0}$:},label=below:$0$] {}
                (1,0) node[label=below:$1$] {}
                (0,1) node [label=left:{$\ppp[\pi]{1}$:}] {}
                (0,2) node [label=left:{$\ppp[\pi]{2}$:}] {}
                (0,3) node [label=left:{$\ppp[\pi]{3}$:}] {};
        \end{scope}
        \begin{scope}[shift={(1.8,0)}]
          \draw[{[-}, red] (0,0)--(.55,0);
          \draw[{-]}] (.55,0)--(1,0);
          \draw[{[-}, green] (0,1)--(.55,1);
          \draw[{-]}, red] (.55,1)--(1,1);
          \draw[{[-},blue] (0,2)--(.55,2);
          \draw[{-]}, green] (.55,2)--(1,2);
          \draw[{[-},purple] (0,3)--(.55,3);
          \draw[{-]},blue ] (.55,3)--(1,3);
          \draw[dotted,thick] (.55,-.5) -- (.55,3.5);
          \path (.1,0) node[vert,red!80!black] {} (.3,1) node[vert,green!80!black] {}
                (.5,1) node[vert,green!80!black] {}  (.27,2) node[vert,blue!80!black] {};
          \path (.4,3) node[vert,purple!80!black] {};
          \path (.8,3) node[vert,blue!80!black] {} (.92,1) node[vert,red!80!black] {}
                (.7,0) node[vert] {} (.65,1) node[vert,red!80!black] {};
          \path (0,0) node [label=left:{$\ppp[\pi']{0}$:},label=below:$0$] {}
                (1,0) node[label=below:$1$] {}
                (0,1) node [label=left:{$\ppp[\pi']{1}$:}] {}
                (0,2) node [label=left:{$\ppp[\pi']{2}$:}] {}
                (0,3) node [label=left:{$\ppp[\pi']{3}$:}] {};
        \end{scope}
      \end{tikzpicture}
    \caption{If $\pi$ is an abstracted permutation and $\pi'$ is its child given by bumping
    the middle point $x$ in $\ppp[\pi]{0}$, then the point processes $\ppp[\pi']{k}$ is equal
    to $\ppp[\pi]{k+1}$ on $[0,x)$ and is equal to $\ppp[\pi]{k}$ on $(x,1]$, as depicted
    above.}\label{fig:down.shift}
  \end{figure}

  \subsection{Moving towards the base in the tree}
  In the previous section, we gave a loose account of how to define the descendants of an abstracted
  permutation~$\pi$ in the limit tree. We need to define the entire tree, however, which includes
  the ancestors of $\pi$ and their descendants.

  Suppose $\pi$ is a (non-abstracted) permutation, and we would like to determine
  both how many children and how many siblings it has in the forest. Again, we ask the question
  of what information about $\pi$ we need to find this out.

  As before, we need to know the locations of $0$-separated letters in $\pi$. We also need to know
  the locations of $-1$-separated letters, which can become $0$-separated letters in the parent of $\pi$.
  Finally, we need to know the value of $\pi(1)$, as this determines the ancestor of $\pi$ in the forest.

  \begin{example}\label{ex:4}
    Let $\pi$ be the permutation from Example~\ref{ex:1}, which has
    $0$-separated letters at positions~$7$ and $27$. Suppose that
    it has $-1$-separated letters at $15$ and $36$, and suppose that
    $\pi(1)=20$.

    As before, $\pi$ has two children in the forest, given by bumping positions~$7$ and $27$.
    When we move towards the base in the tree to the parent of $\pi$, the $0$-separated letter at position~$7$
    becomes a $1$-separated letter at position~$6$, and the $-1$-separated letter at position~$15$ becomes
    a $0$-separated letter at position~$14$, while the separated letters after position~$20$ remain
    the same. The permutation also has a new $0$-separated letter at position~$20$, which if bumped
    leads to $\pi$.  Thus the parent of $\pi$ has three  children total, and $\pi$ has two siblings.
  \end{example}

  Again, we can view this in terms of words encoding the $k$-separated letters.
  We can view the permutation~$\pi$ of Example~\ref{ex:4} as the word $02\red{\bar{1}}1|0\red{\bar{1}}$,
  with the $|$ symbol specifying the value of $\pi(1)$. When moving towards the base in the tree,
  a $0$ is inserted in the position of the $|$ symbol, and all values to the left of the $|$ are incremented.
  See Figure~\ref{fig:ex4} for a depiction of Example~\ref{ex:4} in these terms. As before, this picture is incomplete:
  if the first character in the word corresponds to a separated letter at position~$1$, this character is
  deleted rather than incremented when moving towards the base in the tree. This will be irrelevant
  in the limit, since for a fixed~$K$ and $r$, a random permutation is vanishingly unlikely
  to have a $k$-separated letter with $\abs{k}\leqslant K$ occurring in the first $r$
  positions.\looseness=1

  The extra ingredient in moving towards the base in the tree rather than towards the leaves is knowledge of $\pi(1)$.
  Looking back at Figure~\ref{fig:down.shift}, to go backward from $\pi'$ to $\pi$ in the limit
  tree, we need the location of the dotted line, which cannot be determined from
  $\ppp[\pi']{k}$. In the limit case, these locations will be uniform over $[0,1]$ and independent
  of the point processes.

  \begin{figure}
    \begin{center}
      \begin{tikzpicture}
        \path (0,0) node (pi) {$02\red{\bar{1}}1|0\red{\bar{1}}$}
              (-.75,1) node (c1) {$|2\red{\bar{1}}10\red{\bar{1}}$}
              (.75,1) node (c2) {$\red{\bar{1}}1\red{\bar{2}}0|\red{\bar{1}}$}
              (0,-1) node (parent) {$130200\red{\bar{1}}$}
              (1.5,0) node (sib2) {$02\red{\bar{1}}1\red{\bar{1}}|\red{\bar{1}}$}
              (-1.5,0) node(sib1) {$02|200\red{\bar{1}}$}
            ;
            \draw (parent.north)+(.45,0)--(sib2)
                  (parent.north)+(.15,0)--(pi)
                  (parent.north)+(-.27,0)--(sib1)
                  (pi.north)+(-.55,0)--(c1)
                  (pi.north)+(.37,0)--(c2)
              ;
      \end{tikzpicture}
    \end{center}
    \caption{The children, parent, and siblings of the permutation~$\pi$
    given in Example~\ref{ex:4}.}\label{fig:ex4}
  \end{figure}
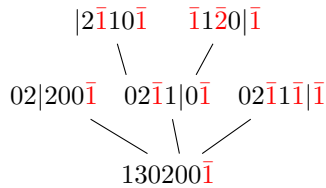

  \subsection{Comparison to other known processes}  \label{sec:comparison}

  It is a natural question to compare the fixed point forest or its subtrees to other known random trees, such
  as the Galton--Watson tree~\cite{GaltonWatson}. The tree component of the fixed point forest containing the
  identity permutation has approximately $(n-1)!$ vertices. As we discussed earlier, the height of this
  component is $2^{n-1}-1$ as shown in~\cite{McKinley}. The Galton--Watson tree on the other hand has
  height $\sqrt{N}$ if the tree has $N$ vertices~\cite{Aldous.1991}, which is much larger than $2^{n-1}$ when
  $N=(n-1)!$. This is consistent with the fact that the fixed point tree has offspring sizes that are correlated across
  generations, whereas the Galton--Watson tree has independent siblings.
  The fixed point forest locally also is quite different from a Galton--Watson tree. For example,
  no leaf in the fixed point forest has a sibling that is also a leaf.

  The process of picking a fixed point at random and moving it to the front (which corresponds to a walk to
  a leaf) has some resemblance to the Tsetlin library~\cite{Tsetlin, Hendricks.1972, Hendricks.1973}, which is a
  model for the evolution of an arrangement of books in a library shelf over time. It is a Markov chain on
  permutations, where the entry in the $i$-th position is moved to the front with probability $p_i$. However,
  in the fixed point forest this process eventually stops when a derangement is reached, whereas in the Tsetlin
  library the process can go on arbitrarily long.

  \section{The limiting objects}

  In this section, we provide the precise definition of the limiting tree.

  \subsection{Local weak convergence in general}\label{sec:lwc.def}
  The main result of this paper is the \defn{local weak convergence} of the fixed point forest to
  a certain limiting tree. This mode of convergence is sometimes called \defn{Benjamini--Schramm convergence}
  after the paper \cite{BS}. We will give a short introduction to local weak convergence now,
  but see \cite[Section~2]{AS} for a more in depth discussion.

  Let $G, G_1,G_2,\ldots$ be a sequence of random rooted graphs.
  For any rooted graph $H$, $H(r)$ denotes the $r$-ball around the root of $H$; that is,
  $H(r)$ is the subgraph of $H$ induced by all vertices at distance~$r$ or less from
  the root. We write $H\cong H'$ to signify that
  $H$ and $H'$ are isomorphic as rooted graphs.
  We say that $G$ is the \defn{local weak limit} of $G_n$ if for every $r\geq 0$ and every finite
  graph $H$,
  \begin{align*}
    \P[G_n(r) \cong H] \to \P[G(r) \cong H]
  \end{align*}
  as $n\to\infty$. Roughly speaking, this says that the view from the root of $G_n$ resembles the view from
  the root of $G$ in distribution more and more as $n\to\infty$.
  Frequently, $G_n$ is a finite, deterministic graph with its root chosen uniformly at random,
  as will be the case in this paper.

  \subsection{Construction of the limit tree}

  The ingredients of our construction are a collection of Poisson point processes
  $(\ppp[\rho]{k})_{k\in\ZZ}$  of unit intensity on $[0,1]$ and a sequence $U_1,U_2,\ldots$ of
  independent $\Unif[0,1]$ random variables.
  Formally, a \defn{point process} is an integer-valued random measure on the Borel sets of $\mathbb{R}$.
  One should think of it as a random collection of points, represented as the atoms
  of the measure.
  A \defn{Poisson point process} $\xi$ of unit intensity on $[0,1]$ is characterized by two
  properties: First, for any interval of length~$x$, the number of points of $\xi$ in
  the interval is distributed as $\Poi(x)$. Second, the numbers of points of $\xi$ in
  disjoint intervals are independent.
  We use the terminology \defn{point process configuration} to mean a deterministic
  collection of points, also represented formally as a measure.

  The point process $\ppp[\rho]{k}$ represents the
  $k$-separated letters in the abstracted permutation~$\rho$.
  Let $\rho_1$ be the parent of $\rho_0=\rho$, let $\rho_2$ be the parent of $\rho_1$, and so on.
  The random variable $U_i$ represents
  the $0$-separated letter in $\rho_i$ that was bumped to create $\rho_{i-1}$.

  To construct the tree, we first define maps corresponding to moving forwards and backwards
  from a given vertex.
  \begin{definition}[The forward map $f$]\label{def:forward.map}
    Let $\xi=(\xi_k)_{k\in\ZZ}$ be a collection of point process configurations on $[0,1]$.
    For any atom~$x$ of $\xi_0$, let $f(\xi,x)=(\xi'_k)_{k\in\ZZ}$, where
    \begin{align*}
      \xi_k' &= \xi_{k+1}\big\rvert_{[0,x)} + \xi_k\big\rvert_{(x,1]}.
    \end{align*}
    This is the down-shift operation depicted in Figure~\ref{fig:down.shift}, corresponding
    to moving forwards towards a leaf in the tree from a permutation to one of its children by bumping
    the abstracted fixed point~$x$.
  \end{definition}
  \begin{definition}[The backward map $b$]\label{def:backward.map}
    Let $\xi=(\xi_k)_{k\in\ZZ}$ be a collection of point process configurations on $[0,1]$.
    For any $u\in[0,1]$, let $b(\xi,u)=(\xi'_k)_{k\in\ZZ}$, where
    \begin{align*}
      \xi'_0 &=
       \xi_{-1}\bigr\rvert_{[0,u)} + \xi_0\bigr\rvert_{(u,1]}
                           + \delta_{u},\\
      \xi'_k&= \xi_{k-1}\bigr\rvert_{[0,u)} + \xi_{k}\bigr\rvert_{(u,1]}\qquad \qquad\text{for $k\neq 0$}.
      \end{align*}
    This is the reverse of the forward map, in the following sense:
    if $x$ is an atom of $\xi_0$ and $f(\xi,x)=\xi'$, then $b(\xi',x)=\xi$.
  \end{definition}

  Next, we define a tree by applying $f$ and $b$ to map out the abstracted permutations.

  \begin{definition}\label{def:varphi}
    Given point process configurations $\pppp[\rho]=(\ppp[\rho]{k})_{k\in\ZZ}$ and a sequence
    $u=(u_1,u_2,\ldots)$ of elements of $[0,1]$, we construct a rooted tree
    $\varphi(u,\pppp[\rho])$ as follows.
    We think of each vertex~$v$ of this tree as an abstracted permutation, represented
    by a collection of point processes $\pppp[v]=(\ppp[v]{k})_{k\in\ZZ}$.
    Let $\rho=\rho_0$ be the root of the tree.
    First, we give $\rho_0$ an infinite chain of ancestors $\rho_1,\rho_2,\ldots$.
    Starting with $\pppp[\rho_0]$, which is given to us, we inductively define
    \begin{align*}
      \pppp[\rho_{i+1}] = b\bigl(\pppp[\rho_i],\, u_{i+1}\bigr).
    \end{align*}

    Next, we construct descendants of each $\rho_i$.
    For every atom~$x$ in $\ppp[\rho_i]{0}$ other than $u_i$, give $\rho_i$ a child~$\rho_i(x)$ and
    define
    \begin{align*}
      \pppp[\rho_i(x)] &= f\bigl(\pppp[\rho_i],\,x\bigr).
    \end{align*}
    for all $k$.
    (We avoid doing this with $x=u_i$ since this would just recreate $\rho_{i-1}$.)
    From here on, we proceed inductively, continuing to extend the tree forwards.
    Suppose that $\pppp[v]$ has already been defined.
    For each atom~$x$ in $\ppp[v]{0}$, extend the tree by creating a child~$v(x)$ of $v$, and define
    \begin{align*}
      \pppp[v(x)] &= f\bigl(\pppp[v],\,x\bigr).
    \end{align*}

    We define $\varphi(u,\pppp[\rho])$ as the resulting tree.
    Also, observing that the $r$-neighborhood of the root of the tree depends only
    on $\bigl(\ppp[\rho]{-r+1},\ldots,\ppp[\rho]{r-1}\bigr)$ and on $u_1,\ldots,u_r$, define
    the map $\varphi_r$ by setting
    $\varphi_r\bigl(u_1,\ldots,u_r; \ppp[\rho]{-r+1},\ldots,\ppp[\rho]{r-1}\bigr)$ as the $r$-neighborhood
    of the root of $\varphi(u,\pppp[\rho])$.
  \end{definition}

  Finally, we construct the limit tree $T$.
  \begin{definition}\label{def:T}
    Define $T = \varphi(U,\pppp[\rho])$,
    where $U=(U_1,U_2,\ldots)$ consists of independent $\Unif[0,1]$
    random variables, and
    $\pppp[\rho]=(\ppp[\rho]{k})_{k\in\ZZ}$ consists of independent Poisson point processes
    on $[0,1]$ with unit intensity, and $\pppp[\rho]$ and $U$ are independent of each other.
  \end{definition}

  In the following section, we will prove that $T$ is the local limit of $F_n$ as $n\to\infty$.

  \section{Local weak convergence of the fixed point tree}
  \label{sec:local.weak.convergence}
  Recall that~$i$ is a $k$-separated position in a permutation~$\pi$ if $\pi(i)=i+k$.
  The main idea in this section is that the neighborhood of a permutation $\pi$ in $F_n$ up to distance~$r$
  can typically be reassembled from two pieces of information: the $k$-separated positions
  in $\pi$ for $-r< k< r$, and the values of $\pi(1),\ldots,\pi(r)$. The first piece lets us
  work out the tree forwards (towards leaves) from $\pi$. When we rescale by $1/n$, these locations
  converge to independent Poisson point processes on $[0,1]$.
  The second piece of information lets us move backwards in the tree. With the same rescaling, these random
  variables converge to independent points sampled uniformly from $[0,1]$.
  The two pieces of information converge jointly, as shown in Proposition~\ref{prop:joint_convergence},
  and the weak local convergence of the fixed point tree
  $F_n$ follows easily from this in Theorem~\ref{thm:lwc}.

  While this convergence is what one would expect from well known Poisson approximations
  of fixed points of random permutations (see \cite[Theorem~11]{ChatterjeeDiaconisMeckes.2005},
  for example), it will take some technical work to prove our precise statement. We will use Stein's method
  via size-bias couplings using the framework from \cite{BHJ}, which we introduce now.
  See also \cite[Section~4.3]{Ross} for a more detailed introduction to size-bias couplings.
  The general idea for our purposes is that we have a collection of 0-1 random
  variables, and we would like to show that they are well approximated by independent Poisson
  random variables. If there exist certain couplings described below,
  Stein's method gives a quantitative version of this approximation.
  The bound is given in terms of the covariances of the random variables and does not
  depend on the couplings, once they are shown to exist.

  \begin{condition}
    Let $\I=(I_\alpha)_{\alpha\in\Ii}$ be a collection of 0-1 random variables.
    For each $\alpha\in\Ii$, there is a random vector
    $J_{\omitted\alpha}=(J_{\beta\alpha})_{\beta\in\Ii}$ coupled with $\I$ such that
    \begin{itemize}
      \item $J_{\omitted\alpha}$ is distributed as $\I$ conditioned on $I_\alpha=1$;
      \item we can partition $\Ii$ into disjoint sets
        \begin{align*}
          \Ii = \Ii_\alpha^+\cup \Ii_\alpha^-\cup\{\alpha\}
        \end{align*}
        such that with probability one,
        \begin{align}
          J_{\beta\alpha}&\leqslant I_\beta\quad\text{if $\beta\in\Ii_\alpha^-,$}\label{eq:monotone-}\\
          J_{\beta\alpha}&\geqslant I_\beta\quad\text{if $\beta\in\Ii_\alpha^+$.}\label{eq:monotone+}
        \end{align}
    \end{itemize}
  \end{condition}
  When this condition holds,
  one can estimate the distance between $\I$ and a vector of independent
  Poisson random variables in terms of the covariances of the components of $\I$, with no
  mention of the coupling.
  Here,
  as usual, the covariance between two random variables $X$ and $Y$ is
  \[
     \cov(X,Y) = \E[XY] - \E[X]\E[Y].
  \]
  The bound is on the \defn{total variation distance} between the laws of the random vectors.
  For random variables $X$ and $Y$ taking values in some space~$\mathcal{S}$, this distance is defined as
  \begin{align*}
    \dtv(X,Y) = \sup_{S\subseteq\mathcal{S}}\abs{\P[X\in S] - \P[Y\in S]}.
  \end{align*}
  Another characterization of the total variation distance between the laws of $X$ and $Y$
  is as the minimum of $\P[X\neq Y]$ over all couplings of $X$ and $Y$; see \cite[Section~A.1]{BHJ}.

  \begin{proposition}[Corollary~10.J.1 in \cite{BHJ}]\label{prop:BHJ}
  Assume the coupling condition.
  Let $\Y=(Y_\alpha)_{\alpha\in\Ii}$ be a vector of independent
  Poisson random variables with $\E Y_\alpha=\E I_\alpha$.  Then
  \begin{align}
    \dtv(\I, \Y)&\leqslant \sum_{\alpha\in\Ii}(\E I_\alpha)^2
      +\sum_{\alpha\in\Ii}\sum_{\smash{\beta\in\Ii_\alpha^-}}
      \abs{\cov(I_\alpha,I_\beta)}
      +\sum_{\alpha\in\Ii}\sum_{\smash{\beta\in\Ii_\alpha^+}}\cov(I_\alpha,I_\beta).
      \label{eq:BHJ}
  \end{align}
\end{proposition}

  We will also need the following technical lemma.
  \begin{lemma}\label{lem:poisson_compare}
    Let $\Y=(Y_\alpha)_{\alpha\in\Ii}$ and let $\Z=(Z_\alpha)_{\alpha\in\Ii}$ be vectors of independent
    Poisson random variables. Then
    \begin{align*}
      \dtv(\Y,\Z) &\leqslant \sum_{\alpha\in\Ii} \bigl\lvert \E Y_\alpha - \E Z_\alpha \bigr\rvert.
    \end{align*}
  \end{lemma}
  \begin{proof}
    Suppose that $U$ and $V$ are Poisson with means $a\leqslant b$. Then $U$ and $V$ can be coupled
    by setting $V=U+W$ where $W\sim\Poi(b-a)$ and is independent of $U$.
    By Markov's inequality,
    \begin{align*}
      \P[U\neq V] = \P[W\geqslant 1]\leqslant b-a.
    \end{align*}
    Applying this coupling to $Y_\alpha$ and $Z_\alpha$ for each $\alpha$, we obtain a coupling
    of $\Y$ and $\Z$ in which they differ with probability at most
    $\sum_{\alpha\in\Ii} \bigl\lvert \E Y_\alpha - \E Z_\alpha \bigr\rvert$.
  \end{proof}

  Next, we apply Proposition~\ref{prop:BHJ}
  to some indicators derived from a random permutation $\pi$ on $[n]:=\{1,\ldots,n\}$
  with distribution to be specified.
  Let $I(i,k)$ be an indicator on position~$i$ being $k$-separated in $\pi$.
  Fix $r$ and $n$, and let
  \begin{align*}
    \Ii = \Bigl\{(i,k)\;\bigl\vert\; \text{$k\in\{-r+1,\ldots,r-1\}$,\, $i\in\{r+1,\ldots,n\}$,
                and $i+k\in\{1,\ldots,n\}$} \Bigr\},
  \end{align*}
  which is the set of $(i,k)$ such that $i>r$ and $i$ might possibly be $k$-separated in $\pi$,
  for $\abs{k}\leqslant r-1$.

  \begin{proposition}\label{prop:discrete_poisson}
    Let $\pi$ be a uniformly random permutation on $[n]$ conditioned on
    \begin{align}
      \pi(1) &= a_1,\;\ldots,\;\pi(r)=a_r\label{eq:conditioning}
    \end{align}
    for some set of $r$ distinct values $A=\{a_1,\ldots,a_r\}\subseteq [n]$.
    Let $\I = \bigl(I(i,k)\bigr)_{(i,k)\in\Ii}$ be the vector of indicators defined above.
    Let $\Z = \bigl(Z(i,k)\bigr)_{(i,k)\in\Ii}$ where
    the components of $\Z$ are drawn independently from $\Poi(1/n)$.
    Then
    \begin{align*}
      d_{TV}(\I,\Z) &\leqslant \frac{16r^2 +2r}{n-r-1}.
    \end{align*}
  \end{proposition}
  \begin{proof}
    The proof proceeds in three steps: First, we construct a coupling satisfying
    the coupling condition. Next, we bound the expression on the right hand side
    of \eqref{eq:BHJ} to obtain a total variation bound between $\I$ and a random vector
    $\Y = \bigl(Y(i,k)\bigr)_{(i,k)\in\Ii}$ whose components are independent Poisson random
    variables with $\E Y(i,k) = \E I(i,k)$. Last, we bound the total variation distance between
    $\Y$ and $\Z$.

    One thing to observe before we start is that if $i+k\in A$, then
    $I(i,k)=0$ deterministically for $(i,k)\in\Ii$.
    The pairs $(i,k)$ where this holds are irrelevant when we bound the
    distance between $\I$ and $\Y$, as the corresponding term in $\Y$ is also deterministically zero.
    Thus we can ignore these terms in Steps~1 and~2 by removing them from $\I$ and $\Y$.
    Let $\Ii'=\{(i,k)\in\Ii\mid i+k\notin A\}$. In a slight abuse of notation, we take $\I$ and $\Y$
    to be indexed by $\Ii'$ rather than by $\Ii$ in Steps~1 and~2 only.

    \step{1}{Constructing the coupling.}

    \noindent
    Fix some $(i_0,k_0)\in\Ii'$.
    Our goal is to construct $(J(i,k))_{(i,k)\in\Ii'}$
    distributed as $\I$ conditioned on $I(i_0,k_0)=1$ and to partition $\Ii$
    so that \eqref{eq:monotone-}--\eqref{eq:monotone+} hold.
    (In the notation used in the coupling condition, $J(i,k)$ would be written $J_{(i,k)(i_0,k_0)}$.
    We omit mention of $(i_0,k_0)$ to simplify notation.)

    Let $\tau$ be the random swap $(\pi(i_0),i_0+k_0)$, and let
    $\pi'=\tau\circ \pi$. This forces $\pi'$ to map $i_0$ to $i_0+k_0$, making
    $i_0$ a $k_0$-separated point for $\pi'$. As $\pi(i_0)$ cannot be an element of $A$ (because $i_0>r$)
    and $i_0+k_0$ is not in $A$ by definition of $\Ii'$, the permutation $\pi'$ also satisfies
    \eqref{eq:conditioning}.

    We show now that $\pi'$ is distributed as $\pi$ conditioned
    on mapping $i_0$ to $i_0+k_0$. Let $\Pi$ be the set of permutations on $n$ elements
    satisfying the conditions specified in \eqref{eq:conditioning}, and let $\Pi'\subseteq\Pi$
    be the set of permutations that also map $i_0$ to $i_0+k_0$. One can easily check that for
    any $\sigma'\in\Pi'$, there are exactly $n-r$ permutations $\sigma\in\Pi$ such that
    swapping $\sigma(i_0)$ and $i_0+k_0$ yields $\sigma'$. As $\pi$ is distributed uniformly
    over $\Pi$, this implies that $\pi'$ is distributed uniformly over $\Pi'$, which
    shows that $\pi'$ is distributed as $\pi$ conditioned on mapping $i_0$ to $i_0+k_0$.

    Now, for $(i,k)\in \Ii'$ we define
    \begin{align*}
      J(i,k) &= \1\{\text{$i$ is $k$-separated in $\pi'$}\},
    \end{align*}
    and the distribution of $\bigl(J(i,k)\bigr)_{(i,k)\in\Ii'}$ is as we wanted.
    To partition $\Ii'$, we define $\Ii^-_{i_0,k_0}$ as all pairs $(i,k)\in\Ii'$ where either
    \begin{enumerate}[label = (\roman*)]
      \item $i=i_0$ and $k\neq k_0$; or \label{item:wrongtarget}
      \item $i=i_0+k_0-k$ and $k\neq k_0$. \label{item:wrongimage}
    \end{enumerate}
    Define $\Ii^+_{i_0,k_0}$ to be the rest of $\Ii'$ except for $(i_0,k_0)$.
    For any $(i,k)\in\Ii^-_{i_0,k_0}$, it is impossible that $\pi'(i) = i+k$:
    If \ref{item:wrongtarget} holds, then we already know that $\pi'(i) = i+k_0$, since we have
    conditioned $\pi'$ to make this so.
    If \ref{item:wrongimage} holds, then we already know that $\pi'(i_0)=i_0+k_0$, and so it cannot be
    that $\pi'(i)=i+k=i_0+k_0$. Thus, for $(i,k)\in\Ii^-_{i_0,k_0}$, we have $J(i,k)=0$, and the coupling
    satisfies \eqref{eq:monotone-}.

    To see that \eqref{eq:monotone+} is satisfied, suppose that $J(i,k)=0$ and $I(i,k)=1$ for some
    $(i,k)\in\Ii'$.
    We will show that $(i,k)\in\Ii^-_{i_0,k_0}$. By our assumption, $\pi(i)=i+k$, and $\pi'(i)\neq i+k$.
    Thus $\tau$ swaps $i+k$ with some other value. By the definition of $\tau$, either
    $i+k=\pi(i_0)$, or $i+k=i_0+k_0$. In the first case, we have $\pi(i_0)=\pi(i)$, implying
    that $i_0=i$; thus $(i,k)$ satisfies \ref{item:wrongtarget}. In the second case, we have
    $i=i_0+k_0-k$, satisfying \ref{item:wrongimage}. Thus our coupling satisfies \eqref{eq:monotone+}.

    \step{2}{Bounding $\dtv(\I,\Y)$.}

    \noindent
    The conditions of Proposition~\ref{prop:BHJ} are now satisfied, and we just need to bound the three terms
    on the right hand side of \eqref{eq:BHJ}.
    We start with the observation that $\pi$ can be thought of as a uniformly random bijection
    from $\{r+1,\ldots,n\}$ to $\{1,\ldots,n\}\setminus A$. Thus, for any $(i,k)\in\Ii'$,
    the probability that $\pi$ maps $i$ to $i+k$ is $1/(n-r)$. Equivalently,  $\E I(i,k)=1/(n-r)$.

    Now, we bound the first term in \eqref{eq:BHJ}.
    As $\abs{\Ii'}\leqslant 2r(n-r)$, we have
    \begin{align}
      \sum_{(i,k)\in\Ii'} (\E I(i,k))^2 &\leqslant \frac{2r}{n-r}. \label{eq:term1}
    \end{align}

    For the next term, observe that if $(i,k)\in\Ii^-_{i_0,k_0}$, then $I(i,k)$
    and $I(i_0,k_0)$ cannot simultaneously be $1$. Thus
    \begin{align*}
      \cov\bigl(I(i,k),\,I(i_0,k_0)\bigr) = -\E I(i,k) \E I(i_0,k_0)=-1/(n-r)^2.
    \end{align*}
    For any $(i_0,k_0)$, the number of pairs $(i,k)$ satisfying \ref{item:wrongtarget} is at most $2r$,
    and the number satisfying \ref{item:wrongimage} is at most $2r$. Thus
    $\bigl\lvert\Ii^-_{i_0,k_0} \bigr\rvert\leqslant 4r$, and
    \begin{align}
      \sum_{(i_0,k_0)\in\Ii'}\sum_{(i,k)\in\Ii^-_{i_0,k_0}} \bigl\lvert\cov\bigl(I(i,k),\,I(i_0,k_0)\bigr)\bigr\rvert
        &\leqslant \frac{2r(n-r)(4r)}{(n-r)^2}=\frac{8r^2}{n-r}.\label{eq:term2}
    \end{align}

    For the final term, suppose that $(i,k)\in\Ii^+_{i_0,k_0}$.
    As $\pi$ conditioned on $I_{i_0,k_0}=1$ is a uniformly random element of the set $\Pi'$ from
    Step~1, we have $\E[I_{i,k}\mid I_{i_0,k_0}=1]=1/(n-r-1)$. Thus
    \begin{align*}
      \cov\bigl(I(i,k),\,I(i_0,k_0)\bigr) &= \E\bigl[I(i,k)I(i_0,k_0)\bigr] - \frac1{(n-r)^2}\\
        &=  \frac{1}{n-r}\E\bigl[I_{i,k}\mid I_{i_0,k_0}=1\bigr]-\frac1{(n-r)^2}\\
        &=\frac{1}{(n-r)(n-r-1)}-\frac{1}{(n-r)^2}=
        \frac{1}{(n-r)^2(n-r-1)}.
    \end{align*}
    Using the bound $\bigl\lvert\Ii^+_{i_0,k_0}\bigr\rvert\leqslant\abs{\Ii'}\leqslant 2r(n-r)$,
    \begin{align}
      \sum_{(i_0,k_0)\in\Ii'}\sum_{(i,k)\in\Ii^+_{i_0,k_0}} \cov\bigl(I(i,k),\,I(i_0,k_0)\bigr)
        &\leqslant \frac{\bigl(2r(n-r)\bigr)^2}{(n-r)^2(n-r-1)} = \frac{4r^2}{n-r-1}.\label{eq:term3}
    \end{align}
    Summing \eqref{eq:term1}--\eqref{eq:term3} and applying Proposition~\ref{prop:BHJ}
    shows that
    \begin{align}
      \dtv(\I,\Y) &\leqslant \frac{12r^2 +2r}{n-r-1}.\label{eq:IYbound}
    \end{align}

    \step{3}{Bounding $\dtv(\Y,\Z)$.}

    \noindent
    If $(i,k)\in\Ii\setminus\Ii'$, then $\E Y(i,k)=0$. If $(i,k)\in\Ii'$, then
    $\E Y(i,k) = 1/(n-r)$. Applying Lemma~\ref{lem:poisson_compare} and using
    the bounds $\abs{\Ii\setminus\Ii'}\leqslant 2r^2$ and $\abs{\Ii'}\leqslant 2r(n-r)$,
    \begin{align*}
      \dtv(\Y,\Z) &\leqslant 2r^2\biggl(\frac1n -0 \biggr) + 2r(n-r)\biggl(\frac{1}{n-r}-\frac{1}{n}\biggr)
        = \frac{4r^2}{n}.
    \end{align*}
    Summing this bound and the one in \eqref{eq:IYbound} proves the theorem.
  \end{proof}

  \begin{proposition}\label{prop:joint_convergence}
    Let $\pi_n$ be a uniform permutation on $n$ elements.
    Let $\pp{k}$ be the point process on $[0,1]$ with atoms at $i/n$ for every position~$i$
    that is $k$-separated in $\pi_n$, and let $\bigl(\ppp{i}\bigr)_{i\in\ZZ}$ be
    independent Poisson point processes on $[0,1]$ with unit intensity.
    Let $U_1,U_2,\ldots$ be distributed uniformly on $[0,1]$, and let them be independent
    of each other and of $\bigl(\ppp{i}\bigr)_{i\in\ZZ}$.
    For any fixed $r$,
    \begin{align}
      \biggl(\frac{\pi_n(1)}{n},\ldots,\,\frac{\pi_n(r)}{n},\,\pp{-r+1},\ldots,\pp{r-1}\biggr)
       \toL \bigl(U_1,\ldots,U_r,\,\ppp{-r+1},\ldots,\ppp{r-1}\bigr)\label{eq:joint_convergence}
    \end{align}
    as $n\to\infty$.
  \end{proposition}
  \begin{proof}
    As a first step, let $\ppr{k}$ be the point process obtained by restricting
    $\pp{k}$ to the interval $\bigl((r+1)/n, 1\bigr]$. For any $i$,
    \begin{align*}
      \P\bigl[\pi_n(i)\in\{i-r+1,\ldots,i+r-1\}\bigr] &\leqslant \frac{2r}{n}.
    \end{align*}
    By a union bound, the probability that any of $1,\ldots,r$ is $k$-separated in $\pi_n$ for some $-r<k<r$
    is at most $2r^2/n$. Thus
    \begin{align*}
      \biggl(\frac{\pi_n(1)}{n},\ldots,\,\frac{\pi_n(r)}{n},\,\pp{-r+1},\ldots,\pp{r-1}\biggr)
        &= \biggl(\frac{\pi_n(1)}{n},\ldots,\,\frac{\pi_n(r)}{n},\,\ppr{-r+1},\ldots,\ppr{r-1}\biggr)
    \end{align*}
    except with probability at most $2r^2/n$.
    Since this probability vanishes as $n\to\infty$, it suffices to show the convergence in distribution
    of the right hand side of the above equation to the limit in~\eqref{eq:joint_convergence}.

    Now, observe that
    \begin{align}
      \biggl(\frac{\pi_n(1)}{n},\ldots,\,\frac{\pi_n(r)}{n}\biggr) \toL
        \bigl(U_1,\ldots,U_r\bigr)
    \end{align}
    by directly calculating
    \begin{align*}
      \lim_{n\to\infty}\P\biggl[\biggl(\frac{\pi_n(1)}{n},\ldots,\,\frac{\pi_n(r)}{n}\biggr)\in
      E_1\times\dots\times E_r\biggr]
    \end{align*}
    for any intervals $E_1,\ldots,E_n$.
    To finish off the theorem, we will show that the law of
    $\bigl(\ppr{-r+1},\ldots,\ppr{r-1}\bigr)$ conditional on
    $\pi_n(1),\ldots,\pi_n(r)$ converges weakly to the law of $(\ppp{-r+1},\ldots,\ppp{r-1})$.

    Observe that $\bigl(\ppr{-r+1},\ldots,\ppr{r-1}\bigr)$ is a function of $\I$, with its entries given by
    \begin{align*}
      \ppr{k} &= \sum_{i=r+1}^n I(i,k)\delta_{i/n}.
    \end{align*}
    Let $\zeta_k^{(n)}=\sum_{i=r+1}^n Z(i,k)\delta_{i/n}$, where $\bigl(Z(i,k)\bigr)_{(i,k)\in\Ii}$ has
    entries independent and distributed as $\Poi(1/n)$.
    Since $\dtv(f(X),f(Y))\leqslant \dtv(X,Y)$ for any random variables $X$ and $Y$ and measurable map~$f$,
    Proposition~\ref{prop:discrete_poisson} implies that the total variation distance between
    $\bigl(\ppr{-r+1},\ldots,\ppr{r-1}\bigr)$ conditional on $\pi_n(1),\ldots,\pi_n(r)$
    and $\bigl(\zeta^{(n)}_{-r+1},\ldots,\zeta^{(n)}_{r-1}\bigr)$
    vanishes as $n\to\infty$.

    Thus we only need to show that
    $\bigl(\zeta^{(n)}_{-r+1},\ldots,\zeta^{(n)}_{r-1}\bigr)$, a collection of discretized
    Poisson point processes,  converges in distribution to the continuous
    Poisson point processes $\bigl(\ppp{-r+1},\ldots,\ppp{r-1}\bigr)$.
    By the independence of entries of
    both of these vectors, we only need to show that $\zeta^{(n)}_k\toL\ppp{k}$ as $n\to\infty$.
    To prove this, by Theorem~16.16 and Chapter~16, Exercise~11 in \cite{Kallenberg},
    it suffices to show that $\zeta^{(n)}_k(B)\toL \ppp{k}(B)$ for any finite union of intervals $B$.
    Both $\zeta^{(n)}_k(B)$ and $\ppp{k}(B)$ are Poisson, and so the convergence holds as
    the portion of points of the form $i/n$ with $i\geqslant r+1$ contained in $B$ approaches the length
    of $B$.
  \end{proof}

  Let us consider the shift operation from the perspective of the point processes
  $\bigl(\pp{k}\bigr)_{k\in\ZZ}$ of Proposition~\ref{prop:joint_convergence}.
  Suppose that $\pi_n$ has a fixed point at position~$i$ and we
  bump it to the beginning of the permutation, moving forward in the tree.
  At times after $i/n$, all point processes $\pp{k}$ are unchanged. At times before $i/n$,
  every point in $\pp{k}$ becomes a point in $\pp{k-1}$ shifted to the right by $1/n$.
  Additionally, $\pp{0}$ loses its point at $i/n$, and $\pp{i-1}$ gains a point at $1/n$.
  This completely describes how the point processes change when this fixed point is shifted.

  \begin{theorem}\label{thm:lwc}
    As $n\to\infty$, the randomly rooted graph $F_n$ converges to $T$ in the local
    weak sense.
  \end{theorem}
  \begin{proof}
    Fix some integer $r\geqslant 1$, and assume throughout that $n>r$.
    We need to show that the $r$-ball around $\pi_n$
    in $F_n$ converges in distribution to the $r$-ball around the root in $T$, which is
    \begin{align*}
      \varphi_r\bigl(\,U_1,\ldots,U_r;\, \ppp{-r+1},\ldots,\ppp{r-1}\bigr).
    \end{align*}
    Let
    \begin{align*}
      T_n&=
      \varphi_r\biggl(\frac{\pi_n(1)}{n},\ldots,\,\frac{\pi_n(r)}{n};\, \pp{-r+1},\ldots,\pp{r-1}\biggr).
    \end{align*}
    The idea will be to show that
    this tree $T_n$ is almost the same thing as the $r$-neighborhood
    of $\pi_n$ in $F_n$, and then to apply Proposition~\ref{prop:joint_convergence} to
    show that $T_n$ converges in distribution to the $r$-ball around the root in $T$.
    \begin{claim}
      If $\pp{-r+1},\ldots,\pp{r-1}$ contain no points in the interval $[0,r/n]$, then $T_n$ is identical to
      to the $r$-ball around $\pi_n$ in $F_n$.
    \end{claim}
    \begin{proof}
      The algorithm defining $T_n$ differs from the true $r$-ball around $\pi_n$ in $F_n$ in two ways.
      First, when moving forward in the tree by bumping a fixed point at position~$i$ to position~$1$, no
      point at $1/n$ is inserted into the point process $\pp{i-1}$.
      For this to cause $T_n$ to lack a
      vertex in the $r$-ball around $\pi_n$,
      it must occur that after $s\geqslant 0$ steps backward in the tree from $\pi_n$,
      there is a fixed point at $i$, and then after bumping it one can move forward in the tree
      another $i-1$ times
      to return the fixed point to $i$, and then one can bump the fixed point again, all
      while remaining within the $r$-ball around $\pi_n$. Thus, it is necessary that $s+i< r$.
      Under the conditions of this claim, $\pi_n(1),\ldots,\pi_n(r)>r$,
      since otherwise $\pi_n$ would have a $k$-separated point at position~$i$ for
      some $-r+1\leq k\leq r-1$ and $1\leq i\leq r$. Hence,
      one must move backwards from $\pi_n$ at least $r-i+1$ times to create a fixed point at $i$,
      showing that this circumstance can never occur.

      The second difference in the algorithm is that points in the point processes are not shifted by
      $1/n$ at each step. One could equally well define the algorithm giving $T_n$ only in terms of
      the order of the points in $\pp{-r+1},\ldots,\pp{r-1}$, and doing this one sees that the shifting
      of points by $1/n$ would not change $T_n$.
    \end{proof}

    As $n\to\infty$, the condition in this claim holds with probability approaching $1$, as we showed
    with a union bound in the proof of Proposition~\ref{prop:joint_convergence}. Thus
    we only need to show that $T_n$ converges in distribution to the $r$-ball around the root in $T$, that is,
    that
    \begin{align*}
      \varphi_r\biggl(\frac{\pi_n(1)}{n},\ldots,\,\frac{\pi_n(r)}{n};\, \pp{-r+1},\ldots,\pp{r-1}\biggr)
      \toL \varphi_r\bigl(\,U_1,\ldots,U_r;\, \ppp{-r+1},\ldots,\ppp{r-1}\bigr).
    \end{align*}
    Once we show that $\varphi_r$ is continuous on a set that almost surely contains
    \begin{align*}
      \bigl(U_1,\ldots,U_r;\, \ppp{-r+1},\ldots,\ppp{r-1}\bigr),
    \end{align*}
    Proposition~\ref{prop:joint_convergence} and the
    continuous mapping theorem~\cite[Theorem~2.7]{Billingsley} immediately give us this result.

    We claim that $\varphi$ is continuous at any point
    $\bigl(u_1,\ldots,u_r;\zeta_{-r+1},\ldots,\zeta_{r-1}\bigr)$
    where all atoms of the point process configurations $\zeta_{-r+1},\ldots,\zeta_{r-1}$
    are distinct, all $u_1,\ldots,u_r\in[0,1]$ are distinct from each other and
    any points in the configurations.
    In fact, $\varphi_r$ is constant on a neighborhood of
    $\bigl(u_1,\ldots,u_r;\,\zeta_{-r+1},\ldots,\zeta_{r-1}\bigr)$,
    as a slight perturbation that does not change the order of any of the points
    $u_1,\ldots,u_r$ or the points in $\zeta_{-r+1},\ldots,\zeta_{r-1}$
    does not change the resulting tree.
    As $\bigl(U_1,\ldots,U_r;\, \ppp{-r+1},\ldots,\ppp{r-1}\bigr)$
    has these properties almost surely, this proves the continuity property of $\varphi_r$
    that we needed.
  \end{proof}
  As a consequence of Theorem~\ref{thm:lwc}, any statistic of $\pi_n$ determined by its
  $r$-neighborhood in $F_n$ converges in distribution to the corresponding statistic
  of the limit tree $T$. This excludes many statistics like the distance from $\pi_n$
  to the nearest leaf of $F_n$, which is nearly local but
  cannot be deduced from the $r$-neighborhood of $\pi_n$ for
  any fixed value of $r$. The following corollary addresses this by a truncation argument.
  \begin{corollary}\label{cor:truncation}
    Let $f$ be a function defined on rooted trees, and suppose that $\min(f(T_0), r)$ is determined by
    the $r$-neighborhood of the root of $T_0$, for any rooted tree $T_0$ and any $r\geqslant 0$.
    Then $f(F_n)\toL f(T)$.
  \end{corollary}
  \begin{proof}
    By assumption, we have $\min(f(F_n),r)=\gamma(F_n(r))$, where $F_n(r)$ denotes the $r$-ball
    around the root of $F_n$ and $\gamma$ is some deterministic function on rooted trees.
    By Theorem~\ref{thm:lwc}, as $n\to\infty$ we have $F_n(r)\toL T(r)$ with respect to the discrete
    topology on rooted trees. Hence, by the
    continuous mapping theorem, $\gamma(F_n(r))\toL \gamma(T(r))$,
    and in our original notation, $\min(f(F_n),r)\toL \min(f(T),r)$.
    As $\P[X\leqslant x] = \P[\min(X,r)\leqslant x]$ for $x<r$, it holds
    for any $x$ where the distribution function
    $\P[f(T)\leqslant x]$ is continuous that
    \begin{align*}
      \lim_{n\to\infty}P[f(F_n)\leqslant x] &=
      \lim_{n\to\infty}\P[\min(f(F_n),\,x+1)\leqslant x]\\ &= \P[\min(f(T),\,x+1)\leqslant x]
         = \P[f(T)\leqslant x].\qedhere
    \end{align*}
  \end{proof}

  \section{Combinatorics of paths to leaves}
  \label{sec:combinatorics distributions}

  In the next section, we will be interested in the limiting distributions of the shortest and longest distance of a
  random permutation to a leaf. In preparation for this, we consider the combinatorics related to the
  shortest and longest path to a leaf in this section.

  For a permutation $\pi \in \mathfrak{S}_n$, let $T(\pi)$ be the fixed point tree with $\pi$ as root.
  We call $i$ a \defn{true fixed point} of $\pi$ if $\pi(i)=i$ and $i\neq 1$.

  \begin{proposition}
  \label{proposition.shortest}
  Given $\pi \in \mathfrak{S}_n$, a shortest path from $\pi$ to a leaf in $T(\pi)$ is obtained by bumping the
  rightmost true fixed point at each step.
  \end{proposition}

  \begin{remark} \label{remark.bumping short}
  Note that the procedure of picking the rightmost true fixed point in $\pi \in \mathfrak{S}_n$ and then bumping
  the letter can be reformulated in the following way. Replace each letter in the one-line notation for $\pi$
  by its $k$-separation. Then scan the word from right to left and successively pick the first $k=0,1,2,\ldots$ until
  no further subsequent bigger $k$ can be picked to the left. Picking $0$ in position 1 is not allowed.
  The letters in these positions in the original $\pi$ are the letters that are bumped.

  For example, take $\pi=32415$. The $k$-separation word is given by $201\red{\bar{3}}0$ (where as in
  Figure~\ref{fig:ex1} a negative separation is indicated in red). The selected positions are underlined
  $\underline{2}0\underline{1}\red{\bar{3}}\underline{0}$, meaning that the letters $5,4,3$ are
  bumped in this order:
  \begin{equation}
  \label{equation.shortest}
    32415 \to 53241 \to 45321 \to 34521.
  \end{equation}
  \end{remark}

  \begin{proof}[Proof of Proposition~\ref{proposition.shortest}]
  Let $v=(v_0, v_1, \ldots, v_\ell)$ with $v_0=\pi$ and $v_\ell$ a leaf be a sequence of vertices in which at some step,
  the rightmost true fixed point is not bumped. Let $v_k$ be the last vertex (meaning $k$ is maximal) where
  anything other than the rightmost true fixed point is bumped in the step $v_k\to v_{k+1}$.
  Denote by $v'=(v_k=v'_k,v'_{k+1},\ldots,v'_{\ell'})$ the path in which the rightmost true fixed point is always bumped.
  We will show that $\ell'\leqslant \ell$. Then moving from right to left replacing all
  non-rightmost bumps produces a (weakly) shorter path than $v$, which shows that the path where the rightmost true
  fixed point is always bumped is a shortest path.

  Let $f_1,\ldots, f_a$ be the true fixed points in $v_k=v'_k$ ordered from left to right and suppose that $f_i$ with
  $1\leqslant i <a$ is the fixed point bumped to reach $v_{k+1}$.
  Since $k$ is maximal such that $v_k\to v_{k+1}$ does not bump the rightmost true fixed point, both
  $(v_{k+1},\ldots,v_\ell)$ and $(v'_k,\ldots,v'_{\ell'})$ follow the algorithm outlined in Remark~\ref{remark.bumping short}
  starting at $v_{k+1}$ and $v'_k$, respectively. Note that by Remark~\ref{remark.bumping short}, the length of the path
  formed by always choosing the rightmost true fixed point can be read off directly from the first vertex (in our case $v_{k+1}$
  and $v'_k$, respectively) by determining the longest consecutive sequence $0,1,2,\ldots$ from right to left in the
  corresponding separation word. Now the separation word of $v_{k+1}$ does not differ from that of $v'_k$ in the positions
  to the right of the position of $f_i$. In the positions to the left of the position of $f_i$, the separations of $v_{k+1}$
  are one less than those of $v'_k$ and shifted one position to the right. In addition there is a new separation
  of value $i-1$ in position 1. Hence the chosen $0,1,2,\ldots$ sequence in $v_{k+1}$ is at most one shorter
  than that in $v'_k$. This proves that $\ell'\leqslant \ell$.
  \end{proof}

  \begin{proposition}
  \label{proposition.longest}
  Given $\pi \in \mathfrak{S}_n$, the longest path from $\pi$ to a leaf in $T(\pi)$ is obtained by bumping the
  leftmost true fixed point at each step.
  \end{proposition}

  \begin{proof}
  Let $v=(v_0, v_1, \ldots, v_\ell)$ with $v_0=\pi$ and $v_\ell$ a leaf be a sequence of vertices in which at some step,
  the leftmost true fixed point is not bumped. We will show the existence of a strictly longer path to a leaf.
  Let $v_k$ be the last vertex such that from $v_k \to v_{k+1}$ anything but the leftmost true fixed point is bumped.
  Denote by $f_1,\ldots, f_a$ the true fixed points in $v_k$ ordered from left to right and suppose that $f_i$ with $i>1$
  is the fixed point bumped in the transition $v_k\to v_{k+1}$.

  We are going to construct a new path to a leaf $v'=(v_0',\ldots,v'_{\ell+1})$ as follows. First, $v_i'=v_i$ for all
  $0 \leqslant i \leqslant k$. In the step $v_k'\to v_{k+1}'$ the true fixed point $f_{i-1}$ is bumped. Except for the letters
  $f_{i-1}$ versus $f_i$, $v_{k+1}'$ and $v_{k+1}$ are the same on the first $f_{i-1}$ positions. Hence we can bump
  the same letters in $v$ and $v'$ until a letter beyond position $f_{i-1}$ is bumped in $v_h\to v_{h+1}$ (with $h>k$).
  At this point, the letter $f_i$ is bumped in $v'_h\to v'_{h+1}$, which is still a true fixed point (no letters to the right of
  position $f_{i-1}$ have changed position up to this point in the bumping process).
  The letters in positions weakly between $f_{i-1}+2$ and $f_i$ agree in $v_h$ and $v'_{h+1}$. The letters different
  from $f_{i-1}$ and $f_i$ weakly before position $f_{i-1}+1$ are shifted one position to the right in $v'_{h+1}$
  compared to $v_h$. Denote the sequence of bumped letters in $v_h\to v_{h+1}\to \cdots \to v_\ell$ by
  $R_1L_1R_2 L_2 \ldots$, where $L_i$ are letters in positions weakly before position $f_{i-1}+1$ and $R_i$ are
  letters after position $f_{i-1}+1$. Write $R_i=R_i'r_i$, where $r_i$ is a single letter. Then the sequence
  $R_1'L_1r_1 R_2' L_2 r_2 \ldots$ is a bumpable sequence on $v'_{h+1}$ due to the relative shift of letter in the left
  $f_{i-1}+1$ positions. This indeed shows that $v'$ is strictly longer than $v$.

  Hence the path given by always bumping the leftmost true fixed point is the unique longest path to a leaf.
  \end{proof}

  Note that unlike for the longest path to a leaf, the shortest path is not unique.

  \begin{example}
  \label{example.bumping}
  Take $\pi = 32415$. Then the longest path to a leaf is of length 9:
  \begin{multline*}
    32415 \to 23415 \to 52341 \to 25341 \to 32541 \to 23541 \to 42351 \\
    \to 24351 \to 32451 \to 23451.
  \end{multline*}
  The shortest path to a leaf given by always bumping the rightmost fixed point is given in~\eqref{equation.shortest}.
  Here is another shortest path:
  \[
    32415 \to 23415 \to 52341 \to 45231.
  \]
  \end{example}

   Denote by $\B:=\B(\pi)$ the set of \defn{bumped values} in the longest path from $\pi$ to a leaf.
   For example, $\B(32415) = \{2,3,4,5\}$ as can be seen from Example~\ref{example.bumping}.
   We now provide a characterization of the set $\B$ that enables us to determine if $\pi(i)\in\B$ given only knowledge
   of $\pi(i+1),\ldots,\pi(n)$.

  \begin{lemma}\label{lem:right.to.left}
    For $\pi \in \mathfrak{S}_n$ and $1\leqslant i\leqslant n$,
    \begin{equation}
    \label{equation.condition B}
      \pi(i)\in\B \quad \iff \quad \pi(i)\neq 1 \quad \text{and} \quad 0\leqslant \pi(i) - i \leqslant \#\{j> i\mid \pi(j)\in\B\}.
    \end{equation}
  \end{lemma}

  \begin{proof}
  Let $\pi(i)\in \B$. First we show that then the conditions on the right of~\eqref{equation.condition B} hold.
  Since only true fixed points can be bumped, we certainly have $\pi(i)\neq 1$.
  Note that $\pi(i)-i$ is the $i$-separation of $\pi(i)$ at position $i$. In the longest path to a leaf, a letter in one-line notation
  either stays in its position or moves to the right unless it is bumped. Hence the separation of a letter either remains
  the same or becomes smaller (unless it is bumped). This implies that a letter with a negative separation
  can never be bumped, which shows that $0 \leqslant \pi(i)-i$. Now suppose that $\pi(i)-i=k$, meaning that $\pi(i)$
  is $k$-separated. Since $\pi(i)\in \B$ and hence must be bumped in the path to the leaf, there must be at least $k$ letters to
  the right of $\pi(i)$ that are bumped to move $\pi(i)$ $k$ positions to the right and make it $0$-separated.
  This is precisely the condition $\pi(i) - i \leqslant \#\{j> i\mid \pi(j)\in\B\}$.

  Conversely, suppose that $\pi(i)$ satisfies the conditions on the right of~\eqref{equation.condition B}
  and set $\pi(i)-i=k$. By Proposition~\ref{proposition.longest}, the longest path is obtained by always bumping the
  leftmost true fixed point, and hence any point that becomes $0$-separated will eventually be bumped. Since by
  assumption $0\leqslant k \leqslant \#\{j> i\mid \pi(j)\in\B\}$, there are at least $k$
  points to the right of $\pi(i)$ that get bumped, which makes $\pi(i)$ eventually $0$-separated, and hence $\pi(i)\in \B$.
  \end{proof}

  Using the above ideas, we can give a bound on the length of the longest path to a leaf, which will be useful
  in Section~\ref{sec:distributions}. For $\pi \in \mathfrak{S}_n$, we denote by $\ell(\pi)$ the length of the longest
  path from $\pi$ to a descendent leaf.

  \begin{lemma}\label{LUB}
  Fix $\pi \in \mathfrak{S}_n$ and let $\B:=\B(\pi) = \{b_1,b_2,\ldots, b_k\}$, where $k = |\B|$ and
  $2\leqslant b_1<b_2<\cdots < b_k\leqslant n$. Then
  \begin{equation}
  \label{equation.LUB}
      \ell(\pi) \leqslant 1 + \sum_{m=1}^{k-1}\prod_{i=1}^m\biggl(1+ \frac{1}{b_i-i}\biggr)\leqslant
    k\prod_{i=1}^{k-1}\biggl(1+ \frac{1}{b_i-i}\biggr).
  \end{equation}
  \end{lemma}

  \begin{proof}
  Let $N_i$ denote the number of times $b_i$ is bumped in the longest path in $T(\pi)$ and let
  $M_i = \sum_{j\geqslant i} N_j$. With this notation $\ell(\pi) = M_1.$

  Every time the letter $b_j$ is bumped to the beginning, it needs to be moved to the right one position at a time
  by $b_j-1$ bumps of letters to the right of $b_j$. The position of $b_j$ can only be increased once by bumps
  of smaller letters $b_i<b_j$.  At least $b_j-j$ letters greater than $b_j$ must be bumped for the value $b_j$ to
  return to position $b_j$ and be eligible to be bumped again. Therefore $N_j\leqslant 1+\frac{1}{b_j-j}M_{j+1}$, and hence
  \begin{align*}
    M_j\leqslant 1 + \biggl(1+ \frac{1}{b_j-j}\biggr)M_{j+1}.
  \end{align*}
  By inductively applying this inequality starting from $M_k\leq 1$, we obtain
  \begin{align*}
    M_j &\leqslant 1 + \sum_{m=j}^{k-1}\prod_{i=j}^m\biggl(1+ \frac{1}{b_i-i}\biggr)
  \end{align*}
  for $j=1,\ldots,k$. Thus
  \begin{align*}
    \ell(\pi)=M_1&\leqslant 1 + \sum_{m=1}^{k-1}\prod_{i=1}^m\biggl(1+ \frac{1}{b_i-i}\biggr)\\
      &\leqslant 1 + (k-1)\prod_{i=1}^{k-1}\biggl(1+ \frac{1}{b_i-i}\biggr)
      \leqslant k\prod_{i=1}^{k-1}\biggl(1+ \frac{1}{b_i-i}\biggr).\qedhere
  \end{align*}
  \end{proof}

  The first bound in~\eqref{equation.LUB} is sharp, giving the correct length
  $2^{n-1}-1$ for the identity permutation of length~$n$.

  \section{Limiting distributions and higher moments}
  \label{sec:distributions}
  In this section we compute the limiting distributions of two statistics of the fixed point
  tree $F_n$ using the local weak limit we have constructed in Section~\ref{sec:local.weak.convergence}.
  As usual, $\pi_n$ is a uniformly random permutation of length~$n$. We study two statistics related to leaves:
  \begin{enumerate}
  \item the distance $M_n$ from $\pi_n$ to the nearest leaf descending from it;
  \item the distance $L_n$ from $\pi_n$ to the farthest leaf descending from it.
  \end{enumerate}
  We use $\Geo(q)$ to refer to the geometric distribution on $\{0,1,\ldots\}$ with parameter~$q$,
  the number of failures before the first success in independent trials that are successful
  with probability~$q$. This distribution places probability~$(1-q)^kq$ on $k$ and has mean~$(1-q)/q$.

  \begin{theorem}\label{thm:limits}\
    \begin{enumerate}[label = (\roman*)]
      \item Let $M\sim \Poi(1)$.
        As $n\to\infty$, we have $M_n\toL M$ and $\E M_n^p\to\E M^p$ for all $p>0$.\label{limits.nearest}
      \item Let $L\sim \Geo\bigl(e^{-1}\bigr)$. As $n\to\infty$, we have $L_n\toL L$
        and $\E L_n^p\to\E L^p$ for all $p>0$.\label{limits.farthest}
    \end{enumerate}
  \end{theorem}
  \begin{proof}
    We will prove that $M_n\toL M$ and $L_n\toL L$ in Propositions~\ref{theorem.shortest}
    and~\ref{theorem.longest}, respectively. In Proposition~\ref{longpathmoment},
    we will show that $\sup_n\E M_n^p<\infty$ and $\sup_n\E L_n^p<\infty$ for any $p>0$.
    It is a standard result that this proves
    the convergence of all moments (see \cite[Theorem~4.5.2]{Chung}, for instance).
  \end{proof}

  An interesting open problem is to determine the limiting distribution of
  the number of steps from $\pi_n$ to a leaf when moving randomly towards a leaf.

  Sections~\ref{subsection.shortest} and~\ref{subsection.farthest} are devoted to proving
  the convergence of $M_n$ and $L_n$ to their limiting distributions.
  Both $M_n$ and $L_n$ are functionals of $F_n$ satisfying the criteria of
  Corollary~\ref{cor:truncation}. This corollary then proves that
  $M_n$ and $L_n$ converge in distribution to the corresponding functionals of the limit tree.
  Thus, our task in these sections is to work out the distributions of the distances in $T$
  from the root to the nearest and farthest leaves descending from it.

  Section~\ref{subsection.higher moments} gives our proof that $M_n$ and $L_n$
  are bounded in $L^p$, establishing the convergence of moments.
  We emphasize that this result is not just a technicality.  The reentry phenomenon
  mentioned on Page~\pageref{page:reentry} can cause $L_n$ to be extremely large.
  For example, if $\pi_n$ begins with the string $1\cdots k$, then $L_n\geqslant 2^{k-1}$, with the same
  small letters bumped repeatedly. The probability of such reentry is vanishingly small, making
  it irrelevant to the distributional convergence of $L_n$. But these unlikely events nonetheless
  contribute to the moments of $L_n$, and a priori
  it is plausible that even the expectation of $L_n$ tends to infinity.
  To prove that this is not the case, we are forced to develop
  several bounds that should be useful in future work on the fixed point forest.

  \subsection{Shortest distance to a leaf}
  \label{subsection.shortest}
  Our first lemma is the limiting analogue of Proposition~\ref{proposition.shortest}.

  \begin{lemma}\label{lem:shortwalk}
  A shortest path from the root of $T$ to a descendant leaf is obtained by bumping the rightmost
  abstracted fixed point at each step.
  \end{lemma}

  \begin{proof}
  The statement can be proved in an analogous fashion to Proposition~\ref{proposition.shortest}.
  \end{proof}

  \begin{proposition}
  \label{theorem.shortest}
    The distance $M_n\toL\Poi(1)$ as $n\to\infty$.
  \end{proposition}
  \begin{proof}
    Let $M$ be the distance in the random tree $T$ from the root to its nearest descendent leaf.
    As we mentioned in our summary of Section~\ref{sec:distributions},
    Corollary~\ref{cor:truncation} shows that $M_n\to M$ in distribution, and thus
    we need only show that $M\sim\Poi(1)$.
    By Lemma~\ref{lem:shortwalk}, $M$ is the number of steps taken if we start at the root of $T$
    and bump the rightmost fixed point until no fixed points remain.

    Walking towards a leaf in the tree in this way can be viewed as follows.
    Recall that $\rho$ is the root of $T$. If $\ppp[\rho]{0}$ has no points,
    then the walk is over, and $M=0$. Otherwise, let $X_1$ be the rightmost point
    of $\ppp[\rho]{0}$, and let $v_1$ be the child of $\rho$ corresponding to bumping it.
    Now, if $\ppp[v_1]{0}$ has no points, then the walk is over and $M=1$.
    Otherwise, let $X_2$ be its rightmost point of $\ppp[v_1]{0}$, which is necessarily to the left of $X_1$,
    since $\ppp[v_1]{0} = \ppp[\rho]{1}\rvert_{[0,X_1)} + \ppp[\rho]{0}\rvert_{(X_1,0]}$.
    Continue in this way, making a sequence of vertices $\rho,v_1,\ldots,v_M$ and
    a sequence of points corresponding to them, $X_1,\ldots,X_M$.

    In this procedure, $X_1$ is the rightmost point of $\ppp[\rho]{0}$, then
    $X_2$ is the rightmost point of $\ppp[\rho]{1}\rvert_{[0,X_1)}$, then
    $X_3$ is the rightmost point of $\ppp[\rho]{2}\rvert_{[0,X_2)}$, and so on.
    We can interpret this as follows: We start at $1$, moving backwards in
    a Poisson process until we encounter a point, which takes $\Exp(1)$ time to arrive.
    Then, looking at a different Poisson process,
    we move backwards until we encounter a point, which again will arrive in $\Exp(1)$ time, independent
    of the first arrival. We then continue on in this way, and $M$ is the total number of points encountered
    before we make it backwards to time~$0$.
    This is the same as counting the number of arrivals in a \emph{single}
    Poisson process between time~$0$ and $1$, which is distributed as $\Poi(1)$.
  \end{proof}

  \subsection{Farthest distance to a leaf}
  \label{subsection.farthest}
  The next lemma is the limiting analogue of Proposition~\ref{proposition.longest}.

  \begin{lemma}\label{lem:longwalk}
    The unique longest path from the root of $T$ to a descendent leaf is obtained by bumping the leftmost
    abstracted fixed point at each step.
  \end{lemma}
  \begin{proof}
    Suppose we have some sequence of points bumped in a path from the root of $T$ to a descendent
    leaf in which some point $y$ was not the leftmost when it was bumped.
    We will show the existence of a strictly longer path to a leaf.
    Let $x<y$ be a point that could have been bumped instead of $y$.
    Decompose the original sequence of bumped points as
    \begin{align}\label{eq:badpath}
      PyL_1a_1L_2a_2\cdots a_{r-1}L_r.
    \end{align}
    Capital letters in \eqref{eq:badpath} denote (possibly empty) strings of points and lowercase
    letters denote single points.
    The string $P$ consists of all bumped points prior to $y$. Next, $L_1$ is made up of points
    smaller than $x$, and $a_1$ is the first point
    larger than $x$. Then $L_2$ consists of points smaller than $x$, and $a_2$ is the next
    subsequent point larger than $x$, and so on. Note that almost surely, no point of
    $\ppp[\rho]{0},\ppp[\rho]{1},\ldots$ occurs more than once, so we never need to worry
    about repeated values in these sequences.

    We claim that we can instead bump the following sequence of points:
    \begin{align}\label{eq:goodpath}
      PxL_1yL_2a_1L_3a_2\cdots a_{r-2}L_ra_{r-1}.
    \end{align}
    As this sequence is one longer than \eqref{eq:badpath}, this claim completes the proof.
    Thus, we just need to show that this sequence is in fact \defn{bumpable}, by which we mean
    that when each point is bumped in turn, the next one is an abstracted fixed point.
    Before we do this, we make a key observation: Suppose that $z_1\cdots z_k$
    and $z_1'\cdots z_k'$ are two bumpable sequences, each of which contains the same number of points
    larger than $z$. Then $z_1\cdots z_kz$ is bumpable if and only if $z'_1\cdots z_k'z$ is bumpable.

    Clearly, $Px$ is bumpable. Since all points in $L_1$ are smaller than $x$ and $x<y$, as each point of
    $L_1$ is encountered in \eqref{eq:badpath} and \eqref{eq:goodpath} the same number of larger points
    has already been encountered in each sequence.
    By our observation, the bumpability of $PyL_1$ implies the bumpability of
    $PxL_1$. Next, since $Py$ is bumpable and all points in $xL_1$ are smaller than $y$, the sequence
    $PxL_1y$ is also bumpable.
    Repeating this reasoning, bumpability of $PyL_1a_1L_2$ implies bumpability of
    $PxL_1yL_2$, and  bumpability of $PyL_1a_1$ implies bumpability
    of $PxL_1yL_2a_1$. Continuing in this way, we arrive at the bumpability of \eqref{eq:goodpath}.
  \end{proof}

  The following proof involves a continuous-time Markov chain known as a \defn{Yule process}.
  At state~$k$, it jumps to $k+1$ at rate~$k$, with no other transitions allowed.
  It is well known that if $(Y_t)_{t\geqslant 0}$ is a Yule process starting at $1$, then
  $Y_t-1\sim\Geo\bigl(e^{-t}\bigr)$; see \cite[Section~4.1.D]{KT}, for example.

  \begin{proposition}
  \label{theorem.longest}
    As $n\to\infty$, we have $L_n\toL\Geo\bigl(e^{-1}\bigr)$.
  \end{proposition}
  \begin{proof}
    Let $L$ be the length of the longest path in the limit tree $T$ from the root
    to a descendent leaf.
    By Corollary~\ref{cor:truncation}, we have $L_n\toL L$.
    By Lemma~\ref{lem:longwalk} the random variable $L$ is the length of the path down $T$ given by
    bumping the leftmost fixed point until none remain.

    First, we sketch one way to compute the distribution of $L$.
    Suppose we start at the root~$\rho$ of $T$. With probability~$e^{-1}$,
    there are no points in $\ppp[\rho]{0}$, and $L=0$. If not, we bump the leftmost point $X$
    of $\ppp[\rho]{0}$ to move to a vertex~$v$. The next set of abstracted fixed points is
    $\ppp[v]{0}$, which is equal to $\ppp[\rho]{1}$ on $[0,X)$ and to $\ppp[\rho]{0}$ on $(X,1]$.
    This is a new Poisson point process independent of the past, and so again we stop with probability
    $e^{-1}$ and $L=1$. Continuing on in this way, $L$ is the number of failures before
    the first success in independent trials that succeed with probability $e^{-1}$.

    The above argument works and is the most direct, but we give a different proof now
    whose ideas will be useful in Section~\ref{subsection.higher moments}.
    Let $\B\subseteq[0,1]$ be the set of points bumped in the longest path in $T$
    from the root to a descendent leaf.
    Lemma~\ref{lem:right.to.left} can be restated in this setting: A point $x$ of $\ppp[\rho]{k}$
    is an element of $\B$ if and only if $0\leqslant k\leqslant \abs{\B\cap(x,1]}$.
    Thus, we can progressively build our set $\B$ by the following procedure.
    Start with $\B$ empty. Scan $\ppp[\rho]{0}$ from right to left starting at time~$1$
    until a point is encountered, and add it to $\B$.
    Now, scan $\ppp[\rho]{0}$ and $\ppp[\rho]{1}$ from right to left from this point until
    a point is encountered, and add it to $\B$. Now scan $\ppp[\rho]{0}$, $\ppp[\rho]{1}$,
    and $\ppp[\rho]{2}$, and so on, stopping when we reach time~$0$.

    In this procedure, the first point arrives at rate~$1$,
    since it is the first arrival time of a unit intensity Poisson process. The next point
    arrives at rate~$2$, since it is the first arrival time of two independent unit intensity
    Poisson processes, and the next at rate~$3$, and so on.
    Thus, the size of $\B$ is the number of increases of a Yule
    process in time~$1$, which is distributed as
    $\Geo\bigl(e^{-1}\bigr)$.
  \end{proof}

  It follows from this proposition that the root of $T$ has finitely many descendants almost surely.
  In forthcoming work, this will be investigated further, and it will be shown that
  the expected number of descendants is infinite.

  \subsection{Higher moments}
  \label{subsection.higher moments}

  To bound the moments of $L_n$ and $M_n$, we must leave behind the limit tree~$T$ and work
  directly with the finite fixed point forest. Since $M_n\leqslant L_n$ and we are only looking for upper bounds,
  we will deal exclusively with $L_n$. Our first result gives an upper bound on $L_n$ in terms of
  $\B=\B(\pi_n)=\{b_1,b_2,\ldots,b_k\}$, the set of letters bumped in the longest path from $\pi_n$ to a descendent leaf.
  Since $b_1<b_2<\cdots < b_k$ is increasing, $b_i-i$ is weakly increasing. Fix some $x>0$ and denote by $\B_x$
  the subset of $\B$ such that $b_i - i < x$ for $b_i \in \B_x$.

  \begin{lemma}
  Fix $x>0$. Then
  \begin{equation} \label{simpleLUB}
    L_n \leqslant 2^{|\B_x|}|\B|\left( 1 + \frac{ 1}{x} \right)^{|\B|}.
  \end{equation}
  \end{lemma}

  \begin{proof}
  For $b_i \in \B_x$, we have $1+ \frac 1 {b_i-i} \leqslant 2$ and for $b_i\in \B\backslash \B_x$, we have
  $1+ \frac 1 {b_i -i} < 1+ \frac{1}{x}.$  With this, the bound in~\eqref{simpleLUB} follows directly from
  Lemma~\ref{LUB}.
  \end{proof}

  We will bound $\E L_n^p$ using \eqref{simpleLUB} by showing that $\abs{\B_x}$ and $\abs{\B}$
  are unlikely to be large. Our first result in this direction, interesting in its own right,
  is an exponential tail bound on $\abs{\B}$.
  Proposition~\ref{theorem.longest} exactly computes the distribution of the
  corresponding quantity in the limit case as $\Geo\bigl(e^{-1}\bigr)$.
  Our proof here follows the same intuition, but it will be considerably more difficult.

  \begin{proposition} \label{prop:bsize}
  For some constant $C<1$, it holds for all $k,n\geqslant 0$ that
  \begin{align*}
    \P\bigl[\abs{\B} \geqslant k\bigr] \leqslant C^k.
  \end{align*}
  \end{proposition}

  We will need several preliminary lemmas first. Recall from Lemma~\ref{lem:right.to.left} that $\B$ can
  be constructed by moving leftward through the random permutation~$\pi_n$, successively revealing
  $\pi_n(n),\pi_n(n-1),\ldots,\pi_n(1)$ and tracking the set $\B$ as we go. Reversing the indexing
  so that we can count up instead of down, define
  \begin{align*}
    X_i = \#\{j > n-i \mid \pi_n(j)\in\B \}.
  \end{align*}
  This yields a \defn{pure birth process} (that is, one that either stays the same or increases by one at each step)
  starting at $X_0=0$ and ending at $X_n=\abs{\B}$.
  Our goal is to prove an exponential tail bound for $X_n$ that does not
  depend on $n$.
  In Theorem~\ref{theorem.longest}, we show that the analogous process for the limit tree is a Yule
  process. The process $(X_i)_{0\leqslant i\leqslant n}$ is not as straightforward,
  but the following lemma shows that it approximates a Yule
  process in the sense that it increases with probability approximately
  proportional to its current value.
  \begin{lemma}\label{lem:discrete.yule}
    For $k<n/2$,
    \begin{align}\label{eq:discrete.yule}
      \P[X_{i+1}=X_i+1 \mid X_i=k] &\leqslant \frac{1+k}{n-2k}.
    \end{align}
  \end{lemma}
  \begin{proof}
    For a given permutation $\sigma\in \mathfrak{S}_n$, let $x_i(\sigma)=\#\{j > n-i \mid \sigma(j)\in\B \}$,
    so that $X_i=x_i(\pi_n)$. Taking $i$ and $k$ to be fixed, define
    \begin{align*}
      \mathfrak{T}_\ell &= \{\sigma\in\mathfrak{S}_n\mid x_i(\sigma)=k,\, \sigma(n-i)=n-i+\ell \},\\
      \mathfrak{U} &= \{\sigma\in\mathfrak{S}_n\mid x_i(\sigma)=k\}.
    \end{align*}
    By Lemma~\ref{lem:right.to.left}, the process $x_i(\sigma)$ increases in its next step
    if and only if $0\leqslant \sigma(n-i)-(n-i)\leqslant x_i(\sigma)$ and $\sigma(n-i)\neq 1$, showing that
    \begin{align}\label{eq:increase.prob}
      \P[X_{i+1}=X_i+1 \mid X_i=k] \leqslant \frac{\sum_{\ell=0}^k\abs{\mathfrak{T}_\ell}}{\abs{\mathfrak{U}}}.
    \end{align}
    Now, we compare the sizes of $\mathfrak{T}_\ell$ and $\mathfrak{U}$ by a combinatorial switching argument.
    Suppose that $\sigma\in\mathfrak{T}_\ell$. For any $j\in[n]$, let
    $\sigma^j=(\sigma(n-i), \sigma(j))\circ\sigma$, the permutation given by swapping the values
    at positions~$n-i$ and $j$ in $\sigma$.
    Given $\sigma^j$, $i$, $\ell$, and $n$ but without $\sigma$ or $j$, we can recover
    $\sigma$ by the formula $\sigma=(n-i+\ell,\sigma^j(n-i))\circ\sigma^j$ and $j$ by
    $j=(\sigma^j)^{-1}(n-i+\ell)$. This shows that the map from $\mathfrak{T}_\ell\times[n]\to\mathfrak{S}_n$
    given by $(\sigma, j)\mapsto \sigma^j$ is injective.

    We claim that for any $\sigma\in\mathfrak{T}_\ell$, the
    permutation~$\sigma^j$ falls in $\mathfrak{U}$ if either of the following holds:
    \begin{enumerate}[label = (\roman*)]
      \item $j\leqslant n-i$; \label{i1}
      \item $j> n-i+\ell$ and $\sigma(j)\notin\B(\sigma)$.\label{i2}
    \end{enumerate}
    Indeed, in the first case, $\sigma$ and $\sigma^j$ differ only at locations $n-i$ and smaller.
    Since $x_i(\sigma)$ is determined by $\sigma(n-i+1),\ldots,\sigma(n)$,
    we have $x_i(\sigma^j)=x_i(\sigma)=k$, and hence $\sigma^j\in\mathfrak{U}$.
    In the second case, the only way for $x_i(\sigma^j)$ to be different from $x_i(\sigma)$ is if
    $\sigma^j(j)\in\B$. But this cannot occur, since
    $\sigma^j(j)=\sigma(n-i)\leq n-i+\ell$ and $j>n-i+\ell$.

    For $\sigma\in\mathfrak{T}_\ell$, there are $n-i$ choices of $j$ satisfying \ref{i1}.
    There are another $i-\ell$ choices
    of $j$ satisfying $j>n-i+\ell$; at most $x_{i-\ell}(\sigma)\leqslant x_i(\sigma)=k$
    of these have $\sigma(j)\in\B(\sigma)$,
    giving us at least $i-\ell-k$ choices of $j$ satisfying \ref{i2}. Thus, for each $\sigma\in\mathfrak{T}_\ell$,
    there are at least $n-\ell-k$ choices of $j$ for which $\sigma^j\in\mathfrak{U}$. By injectivity
    of the map $(\sigma,j)\mapsto\sigma^j$,
    \begin{align*}
      \abs{\mathfrak{U}} \geq (n-\ell-k)\abs{\mathfrak{T}_\ell} \geq (n-2k)\abs{\mathfrak{T}_\ell}
    \end{align*}
    for $\ell\leq k$. Substituting this into~\eqref{eq:increase.prob} gives~\eqref{eq:discrete.yule}.
  \end{proof}

  Now, we couple $X_i$ with a true Yule process,
  whose marginal distributions we know to be exactly geometrically distributed.
  The only complication is that we lose control of $(X_i)$ if it becomes too large,
  which we deal with by considering it only up to a stopping time.
  \begin{lemma}\label{lem:yule.coupling}
    Let $Y_t$ be a Yule process starting from $1$.
    Let $S=\min\{i\mid X_i\geqslant n/4\}$, taking $S=n$ if the minimum is over the empty set.
    The processes $(X_i)_{0\leqslant i \leqslant n}$ and $(Y_t)_{0\leqslant t\leqslant 4}$
    can be coupled so that
    \begin{align*}
      1 + X_i \leqslant Y_{4i/n}
    \end{align*}
    for all $i\leqslant S$.
  \end{lemma}
  \begin{proof}
    Essentially, we just need to confirm that $X_i$ is less likely to increase from time~$i$ to $i+1$
    than $Y$ is from time~$4i/n$ to $4(i+1)/n$.
    Formally, we start with $(X_i,\,i\in [n])$ and then build up the Yule process inductively on the same
    probability space.
    Since a Yule process is Markov, we can construct $(Y_t)$ up to some time~$t_0$ and then
    extend it by attaching to its end a new Yule process starting at $Y_{t_0}$.

    Assume for some $j\in\{0,\ldots,n\}$ that we have already constructed
    a Yule process $(Y_t)_{0\leqslant t\leqslant 4j/n}$ so that
    $1+X_i\leqslant Y_{4i/n}$ for $i\leqslant\min(j,S)$. Note that this is trivial
    in the starting case $j=0$. We want to show that if $S>j$, then we can extend $(Y_t)$ up to time
    $4(j+1)/n$ so that $1+X_{j+1}\leqslant Y_{4(j+1)/n}$. (If $S\leqslant j$, then we can just extend $(Y_t)$ up to
    time $4(j+1)/n$ independently from $X_{j+1}$, since the relationship between the two is irrelevant.)

    So, we assume $S>j$, which implies that $X_j<n/4$.
    According to Lemma~\ref{lem:discrete.yule},
    \begin{align*}
      \P[X_{j+1}=X_j+1\mid X_j] \leqslant \frac{1+X_j}{n-2X_j}\leqslant \frac{2(1+X_j)}{n}
    \end{align*}
    under this assumption.
    The next increase of a Yule process at $Y_{4j/n}$ arrives at rate~$Y_{4j/n}$.
    Thus, conditional on $Y_{4j/n}$, the probability that a Yule process increases from $Y_{4j/n}$
    to $Y_{4j/n}+1$ within time~$4/n$ is
    \begin{align*}
      1 - \exp\biggl(-\frac{4Y_{4j/n}}{n}\biggr) \geqslant 1 - \exp\biggl(-\frac{4(1+X_j)}{n}\biggr)
        \geqslant \frac{2(1+X_j)}{n}.
    \end{align*}
    The first inequality uses the inductive hypothesis that $1+X_j\leqslant Y_{4j/n}$, and the second
    uses the inequality $1-e^{-x}\geqslant x/2$, which holds for $0\leqslant x\leqslant 1$.
    Thus, conditional on $X_j$ and $Y_{4j/n}$,
    a Yule process starting at $Y_{4j/n}$ is more likely to increase in time $4/n$
    than is the process $(X_i)$ from time $j$ to $j+1$. Thus, conditionally on $X_j$
    and $Y_{4j/n}$, we can couple a Yule process
    starting at $Y_{4j/n}$ of duration $4/n$ with $(X_i)$ so that it increases only if $(X_i)$
    does from $j$ to $j+1$. Tacking this Yule process onto the end of $(Y_t)$, we have extended
    our coupling as desired.
  \end{proof}

  \begin{proof}[Proof of Proposition~\ref{prop:bsize}]
    Recall that $X_n=\abs{\B}$.
    Couple $(X_i)$ and $(Y_t)$ according to Lemma~\ref{lem:yule.coupling}.
    Fix $k\leqslant n/4$, and let $S'=\min\{i\mid X_i\geqslant k\}$, taking $S'=n$ if the minimum is over
    the empty set. Since $S'\leqslant S$, we have $1+X_{S'}\leqslant Y_{4S'/n}$.

    We claim that if $X_n\geqslant k$, then $Y_4\geqslant 1+k$. To see this, observe that if
    $X_n\geqslant k$, then $X_{S'}\geqslant k$. Thus
    \begin{align*}
      Y_4\geqslant Y_{4S'/n}\geqslant 1+X_{S'}\geqslant 1+k.
    \end{align*}
    This demonstrates that
    \begin{align}\label{eq:smallk}
      \P[X_n\geqslant k] \leqslant \P[Y_4-1\geqslant k] = (1-e^{-4})^k\leqslant \exp\bigl( -e^{-4}k \bigr)
    \end{align}
    since $Y_t-1$ is distributed as $\Geo\bigl(e^{-t}\bigr)$ (see \cite[Section~4.1.D]{KT}, for example).

    For $n/4<k\leqslant n$,
    \begin{align}\label{eq:largek}
      \P[X_n \geqslant k] \leqslant \P[X_n\geqslant n/4] \leqslant\exp\biggl( -\frac{e^{-4}n}{4}\biggr)
      \leqslant \exp\biggl( -\frac{e^{-4}k}{4}\biggr).
    \end{align}
    The right hand side of \eqref{eq:largek} is larger than that of \eqref{eq:smallk}.
    As $X_n$ does not take values larger than $n$, this shows that
    \eqref{eq:largek} holds for all $k\geqslant 0$.
  \end{proof}

  Next, we prove a subexponential tail bound on the size of $\B_x$ for any fixed $x$.

  \begin{lemma}
  \label{bxsize}
  For $x>0$ and $t\leqslant n-x$,
  \begin{equation} \label{eq:bxsize}
    \P[ \abs{\B_x} \geqslant t] \leqslant {t+x \choose t} \prod_{i = 1}^t \left(\frac{x}{n-i}\right)
       \leqslant \biggl(\frac{ex}{t}\biggr)^t.
  \end{equation}
  \end{lemma}

  \begin{proof}
  In order for each of the first $t$ values in $\B$ to satisfy $b_i - i < x$, we must have that all of the values
  are less than $t+x$.  For a specific choice of letters $1<c_1<\cdots<c_t$, let $E(c_1,\ldots,c_t)$ denote the
  event that the first $t$ values of $\B$ are $c_1, \ldots, c_t.$  If $|\B_x| \geqslant t$ then $E(c_1,\ldots, c_2)$
  holds for one of at most ${t +x \choose t}$ possible choices of $(c_1,\ldots,c_t).$

  For a particular choice of letters $c_1<\cdots<c_t<t+x$,
  \[
    \P[E(c_1,\ldots,c_t)] \leq \P\left[ \bigcap_{i=1}^t \{\pi_n^{-1}(c_i)< x+ i\}\right]= \prod_{i=1}^t \frac{ x }{n-i+1}.
  \]

  To prove the final bound in \eqref{eq:bxsize}, we apply the union bound and evaluate
  \begin{align*}
    \binom{t+x}{t} \prod_{i = 1}^t \left(\frac{x}{n-i+1}\right)
      &= \frac{x^t  (x+t)\cdots(x+1)}{t! \;n\cdots(n-t+1)}\leqslant \frac{x^t}{t!},
  \end{align*}
  since $x+t\leqslant n$ under the assumptions of the lemma, and then we apply the bound
  $t!\geqslant (t/e)^t$.
 \end{proof}

 Proposition~\ref{prop:bsize} and Lemma~\ref{bxsize} now combine to bound $\E L_n^p$.

 \begin{proposition} \label{longpathmoment}
   For any $p>0$, it holds that $\sup_n \E L_n^p < \infty$.
 \end{proposition}

 \begin{proof}
   Fix $p>0$, and choose $x$ large enough that $(1+\frac{1}{x})^{2p}<C^{-1}$
   for the constant~$C$ from Proposition~\ref{prop:bsize}.
   By~\eqref{simpleLUB} and the Cauchy--Schwarz inequality,
   \begin{align}\label{eq:cs}
     \E L_n^p \leqslant \E\biggl[ 2^{2p\abs{\B_x}} \biggr]^{1/2}
             \E\Biggl[ \abs{\B}^{2p}\biggl(1 + \frac{1}{x}\biggr)^{2p\abs{\B}} \Biggr]^{1/2}.
   \end{align}
   By Proposition~\ref{prop:bsize},
   \begin{align}\label{eq:exterm}
     \E\Biggl[ \abs{\B}^{2p}\biggl(1 + \frac{1}{x}\biggr)^{2p\abs{\B}} \Biggr]
     &\leqslant \sum_{t=0}^\infty C^tt^{2p}\biggl(1 + \frac{1}{x}\biggr)^{2pt},
   \end{align}
   which is summable.
   Similarly, applying Lemma~\ref{bxsize} and using $\abs{\B_x}\leqslant n$,
   \begin{align*}
     \E\Bigl[ 2^{2p\abs{\B_x}} \Bigr]
       &\leqslant \sum_{t=0}^{n-x}\P\bigl[\abs{\B_x}=t\bigr]2^{2pt} + \P\bigl[\abs{\B_x}\geqslant n-x\bigr]2^{2pn}\\
       &\leqslant \sum_{t=0}^{\infty}\biggl(\frac{ex}{t}\biggr)^t2^{2pt} +
          \biggl(\frac{ex}{n-x}\biggr)^{n-x}2^{2pn}.
   \end{align*}
   The first term is finite, and the last term vanishes as $n\to\infty$.
   Together with \eqref{eq:exterm}, this yields an upper bound on \eqref{eq:cs}
   with no dependence on $n$.
 \end{proof}

 \section{Further Questions}\label{sec:further}

  Many of the open questions in~\cite{McKinley} about the global properties of the fixed point forest $F_n$
  remain unanswered. Recall that each base of $F_n$ is a permutation with $1$ as a fixed
  point. Let $T_n$ denote the tree in $F_n$ with the identity permutation as its base.
  In fact, $T_n$ is the largest tree in $F_n$, which can be seen as follows. Let $\pi$ be any base with
  $\pi(1) = 1.$  Let $i$ be the largest index of such that $\pi(i) \neq i.$  Then $\pi(i)$ is never bumped, so
  switching $\pi(i)$ and $\pi(1)$ creates a new permutation $\pi'$ with subtree at least as big as the tree with base $\pi$.
  Note that $\pi'\in T_n$, since all letters after $1$ are fixed points. Hence an isomorphic copy of the tree starting from
  $\pi$ is also in $T_n$.

  In~\cite[Propositions~24 and~25]{McKinley}, it is shown that for a uniformly random permutation $\pi_n$,
  \begin{align*}
    \frac{1}{n}\leqslant \P[\pi_n\in T_n]\leqslant \frac{e}{n},
  \end{align*}
  and it is conjectured that $\P[\pi_n\in T_n]\sim \frac{1}{n}$. We highlight this question here and tack
  on a few more of our own:
  \begin{question}
    Prove that $n\P[\pi_n\in T_n]$ converges as $n\to\infty$ and determine its limit. Characterize all permutations
    in $T_n$. How do the next largest components compare in size to it?
  \end{question}
  \noindent Another question posed in \cite{McKinley} is:
  \begin{question}
    Let $R_n$ be the distance from $\pi_n$ to the base of its tree in the fixed point forest.
    What are the limiting asymptotics of $\E R_n$?
  \end{question}
  \noindent One could also ask about the limiting fluctuations of $R_n$ from its mean.

  Though some of our work in Section~\ref{sec:combinatorics distributions} could be helpful in addressing
  these questions, our limit tree has nothing to say about them.
  For example, the limiting tree has no base at all, reflecting that questions about the distance from
  $\pi_n$ to a base are not local. One could instead look for
  a different limit of the fixed point forest where edges
  are scaled so that the diameter of each component of the forest stays bounded, along the lines of
  the continuum random tree (see \cite{Aldous.1991}).

  The present work has created other avenues to explore. Continuing in the same
  theme as Section~\ref{sec:distributions}, what can we say about
  paths from root to leaf in the limiting tree besides the longest and shortest ones? In particular,
  \begin{question}
    Walk from the root towards the leaves in the limiting tree by choosing randomly among all children
    at each step. What is the distribution of the number of steps before reaching a leaf?
  \end{question}
  \noindent
  Other properties of the tree are interesting as well:
  \begin{question}
    Is simple random walk on the limiting tree transient or recurrent?
    What is the branching number (see \cite[Section~1.2]{LyonsPeres}) of the tree?
  \end{question}
  \noindent
  A random walk on the nonlimit fixed point forest can be interpreted as a stochastic version of
  the bumping process where we randomly bump and unbump letters of a permutation. A solution to the above
  problem would likely give information on how quickly this process moves from the starting permutation.

  Another problem is to determine what happens when the root of the fixed point forest is chosen
  nonuniformly:
  \begin{question}
    Determine the local limit of the fixed point forest when the root is sampled from the
    Ewens or Mallows distributions.
  \end{question}
  \noindent The fixed points of Mallows distributions are not distributed evenly over the permutation,
  so we would expect convergence to a limiting tree defined by Poisson point processes
  of nonuniform intensity.





\ACKNO{We would like to thank David Aldous, Persi Diaconis, Gwen McKinley, and Graham White for discussions.
This work benefited from computations in \textsc{SageMath}~\cite{combinat,sage}.}


\end{document}